\documentclass[reqno, psamsfonts]{amsart}

\usepackage{amssymb}

\usepackage[centering]{geometry}

\usepackage{tikz-cd}
\tikzcdset{			% small default diagram size
	diagrams={sep=small}
}

\usepackage{mathtools}

\usepackage[colorlinks=true]{hyperref}
\usepackage{cleveref}

\numberwithin{equation}{section}

\newtheorem{theorem}[equation]{Theorem}
\newtheorem{lemma}[equation]{Lemma}
\newtheorem{proposition}[equation]{Proposition}
\newtheorem{corollary}[equation]{Corollary}

\theoremstyle{definition}
\newtheorem{definition}[equation]{Definition}

\theoremstyle{remark}
\newtheorem{remark}[equation]{Remark}
\newtheorem{example}[equation]{Example}

\DeclareMathOperator{\PGL}{PGL}

\DeclareMathOperator{\Aut}{Aut}

\DeclareMathOperator{\id}{id}
\DeclareMathOperator{\im}{im}

\newcommand{\SheafAut}{\underline{\mathrm{Aut}}}

\DeclareMathOperator{\Spec}{Spec}

\DeclareMathOperator{\Supp}{Supp}
\DeclareMathOperator{\Pic}{Pic}
\DeclareMathOperator{\NS}{NS}
\DeclareMathOperator{\Def}{Def}

\DeclareMathOperator{\codim}{codim}
\DeclareMathOperator{\CH}{CH}

\DeclareMathOperator{\ch}{ch}

\DeclareMathOperator{\MW}{MW}

\DeclareMathOperator{\pr}{pr}

\DeclareMathOperator{\Gal}{Gal}

\newcommand{\ZZ}{\mathbb{Z}}
\newcommand{\QQ}{\mathbb{Q}}
\newcommand{\RR}{\mathbb{R}}
\newcommand{\CC}{\mathbb{C}}

\newcommand{\PP}{\mathbb{P}}

\newcommand{\so}{\mathfrak{so}}

\DeclareMathOperator{\Fitt}{Fitt}

\DeclareMathOperator{\Alb}{Alb}

\newcommand{\Kum}{\text{Kum}}

\begin{document}
	\title[Dual Lagrangian fibration of known HK]{The dual Lagrangian fibration of known hyper-K\"ahler manifolds}
	
	\author[Y.-J. Kim]{Yoon-Joo Kim}
	\address{Stony Brook University, Department of Mathematics, Stony Brook, NY 11794-3651}
	\email{yoon-joo.kim@stonybrook.edu}
	
%	\thanks{Thanks}
%	\date{\today}
	
	\begin{abstract}
		Given a Lagrangian fibration $\pi : X \to \mathbb{P}^n$ of a compact hyper-K\"ahler manifold of $\text{K3}^{[n]}$, $\text{Kum}_n$, OG10 or OG6-type, we construct a natural compactification of its dual torus fibration. Specifically, this compactification is given by a quotient of $X$ by certain automorphisms acting trivially on the second cohomology and respecting the Lagrangian fibration. It is a compact hyper-K\"ahler orbifold with identical period mapping behavior as $X$.
	\end{abstract}
	
	\maketitle
%	\tableofcontents

	\section{Introduction}
	Let $Y$ be a compact Calabi--Yau manifold with a fixed K\"ahler class and $\pi : Y \to B$ its Lagrangian fibration. A general fiber of $\pi$ is a torus by the classical Arnold--Liouville theorem. Any torus has its dual, so one may wonder if we can systematically dualize general fibers of $\pi$ to obtain a new fibration $\check \pi$. The mirror symmetry conjecture in \cite{str-yau-zas96} predicts this should be possible for certain situations. More specifically, one expects there exists a ``dual Lagrangian fibration'' $\check \pi : \check Y \to B$ satisfying: (1) $\check Y$ is a compact Calabi--Yau orbifold and $\check \pi$ is its Lagrangian fibration, and (2) the smooth fibers of $\check \pi$ are dual tori to the smooth fibers of $\pi$. When the Calabi--Yau manifold of interest is a K3 surface, there is a holomorphic variant of this question. The K\"ahler class and Lagrangian fibration are replaced into a holomorphic symplectic form and holomorphic elliptic fibration $\pi : X \to B$. Unfortunately, elliptic curves are self-dual, so the original $\pi : X \to B$ satisfies both the conditions (1--2) and the conjecture becomes rather uninteresting.
	
	A compact hyper-K\"ahler manifold is a higher dimensional generalization of a K3 surface. It is a simply connected compact K\"ahler manifold with a unique global holomorphic symplectic form up to scale. Let $\pi : X \to B$ be a holomorphic Lagrangian fibration of a compact hyper-K\"ahler manifold $X$. By the same reasons for K3 surfaces, \cite[\S 2]{gross-tos-zhang13} claimed $\pi$ should be considered as self-dual \emph{if} the following two conditions hold: (A) all the torus fibers of $\pi$ are principally polarized abelian varieties, and (B) $\pi$ admits at least one section.
	The role of the assumption (A) is to say $\pi$ is fiberwise self-dual, and the role of (B) is to single out a uniform dualization of complex tori as a family. If the assumptions are dropped, there is a priori no reason why one should believe the existence of a good notion of a dual Lagrangian fibration $\check \pi : \check X \to B$. The goal of this paper is to give, without the assumptions (A--B), one distinguished candidate of a dual Lagrangian fibration $\check \pi$ that satisfies all the expected properties. Unfortunately, we were able to realize our strategy only for the currently known deformation types of hyper-K\"ahler manifolds (\Cref{thm:introduction}), but we believe similar results should hold in the most general set-up. Once the assumption (A) fails, the construction yields a compact hyper-K\"ahler orbifold $\check X$ that is not homeomorphic to $X$. The technical assumption (B) will be completely overcome.
	
	Let again $X$ be a compact hyper-K\"ahler manifold of dimension $2n$. A \emph{Lagrangian fibration} of $X$ in this paper will mean a holomorphic surjective morphism $\pi : X \to B$ with connected fibers to a complex manifold $B$ of $0 < \dim B < 2n$. By \cite{hwang08} and \cite{greb-lehn14}, the base $B$ is necessarily isomorphic to $\PP^n$. It is well-known that any smooth fiber of $\pi$ is a complex Lagrangian subtorus of $X$, so by restricting the Lagrangian fibration $\pi$ to its smooth locus $B_0 \subset B$ we get a torus fibration $\pi_0 : X_0 \to B_0$, a smooth proper family of complex tori.
	
	The dual Lagrangian fibration $\check \pi$ will be obtained by a suitable compactification of the ``dual torus fibration'' $\check \pi_0 : \check X_0 \to B_0$ which fiberwise dualizes the original torus fibration $\pi_0$. \cite{sawon04} and \cite{nagai05} proposed to define the dual torus fibration as the relative Picard scheme of $\pi_0$. While this definition behaves well when $\pi_0$ admits a section, it behaves slightly awkward when $\pi_0$ has no sections. We thus start with proposing a new definition of $\check \pi_0$. Recall the fact that all the torus fibers of $\pi_0$ are canonically polarized (e.g., Voisin's argument in \cite[Prop 2.1]{cam06}). That is, each torus fiber $F$ of $\pi_0$ admits a natural isogeny $F \to \check F$ to its dual torus $\check F$. Let us denote the kernel of this isogeny by $(\ker)$ and obtain an isomorphism $\check F \cong F / (\ker)$. The idea is to make this discussion global over the entire base $B_0$. In \Cref{thm:abelian scheme}, we will attach a canonically polarized abelian scheme $P_0 \to B_0$ to $\pi_0$ so that $X_0$ becomes a $P_0$-torsor (this combines the results of Arinkin--Fedorov and van Geemen--Voisin). Let $K_0$ be the kernel of this canonical polarization $P_0 \to \check P_0$. It is a group scheme over $B_0$ acting on both $P_0$ and $X_0$. Take the $K_0$-quotient of both spaces; on the one hand we recover the dual abelian scheme $\check P_0 \cong P_0 / K_0$, and on the other hand we obtain a new space
	\[ \check \pi_0 : \check X_0 \to B_0 \qquad\mbox{for}\quad \check X_0 = X_0 / K_0 .\]
	By construction, $\check X_0$ is a $\check P_0$-torsor, a smooth proper family of complex tori which are fiberwise dual to the original fibration $\pi_0$. This $\check \pi_0$ is our definition of the \emph{dual torus fibration}. We will later see that if $\pi_0$ admits at least one section, then this $\check \pi_0$ becomes isomorphic to the relative Picard scheme of $\pi_0$.
	
	It is important to notice that the group scheme $K_0$ is only a finite \'etale group scheme over $B_0$. One can think of this as the total space of a local system on $B_0$; there is a monodromy issue hiding on the background, and \emph{a priori} $K_0$ may not be a constant group scheme. We are now ready to state the main result of this paper.
	
	\begin{theorem} \label{thm:introduction}
		Let $\pi : X \to B$ be a Lagrangian fibration of a compact hyper-K\"ahler manifold. Assume $X$ is of $\text{K3}^{[n]}$, $\Kum_n$, OG10 or OG6-type. Then
		\begin{enumerate}
			\item The kernel group scheme $K_0 \to B_0$ extends to a constant group scheme $K \to B$ that acts on the entire Lagrangian fibration $\pi : X \to B$. Moreover, $K$ is a subgroup\footnote{We will frequently view a finite constant group scheme $K \to B$ as a finite group, and vice versa. We will denote them by the same letter $K$ if no confusions arise.} of the group
			\[ \Aut^{\circ} (X/B) = \{ f \in \Aut(X) : \pi \circ f = \pi, \ \ f^* \mbox{ acts as the identity on } H^2 (X, \ZZ) \} .\]
			
			\item The quotient
			\[ \check \pi : \check X \to B \qquad \mbox{for} \quad \check X = X / K \]
			compactifies the dual torus fibration $\check \pi_0$.
			
			\item $\check X$ is a compact hyper-K\"ahler orbifold and $\check \pi$ is its Lagrangian fibration. Moreover, $\check X$ has the same period mapping/deformation behavior as $X$.
		\end{enumerate}
	\end{theorem}
	
	If $X$ is of $\text{K3}^{[n]}$ or OG10-type, then the group $K$ (or the constant group scheme $K \to B$) is in fact trivial and these hyper-K\"ahler manifolds are self-dual. On the other hand, if $X$ is of $\Kum_n$ or OG6-type, then $K$ is nontrivial and $\check X$ is not even homeomorphic to $X$. We will provide explicit computations for the group $K$ in \Cref{thm:aut0 computation} and \Cref{rmk:polarization scheme equality}. Note also that $\check X$ is a global quotient of $X$ by automorphisms acting trivially on $H^2 (X, \ZZ)$. As a result, the second rational cohomology of $\check X$ and $X$ are isometric as Beauville--Bogomolov quadratic spaces. The higher cohomology of $\check X$ may be strictly smaller than that of $X$ by \cite{ogu20}, but they are still tightly connected via their Looijenga--Lunts--Verbitsky (LLV) structures (see \cite{loo-lunts97}, \cite{ver95} and \cite{gklr22}). Finally, the singularities of $\check X$ are quotient singularities of high codimensions ($\ge 4$), so they do not admit any symplectic resolutions. We briefly recall for reader's convenience the notion of a singular hyper-K\"ahler variety and its Lagrangian fibration in \Cref{sec:singular HK}.
	
	\begin{remark}
		There were several previous results on the constructions of dual Lagrangian fibrations of compact hyper-K\"ahler manifolds. Especially, \cite[Thm 24]{sawon20} announced the construction of a dual Lagrangian fibration of certain $\Kum_n$-type hyper-K\"ahler manifolds (without a proof). Although Sawon's method is different from ours, it is isomorphic to our construction when $\pi$ admits a section and the polarization type is $(1,\cdots,1,n+1)$. This can be shown by using the results in \Cref{sec:aut0 computation}. \cite{sawon04} and \cite{nagai05} discussed a possible hyper-K\"ahler structure on a partial compactification of the relative Picard scheme of $\pi_0$. These are different to our direction because our dual torus fibration $\check \pi_0 : \check X_0 \to B_0$ is not isomorphic to the relative Picard scheme when $\pi_0$ does not have any section. \cite{mar-tik07} and \cite{menet14} introduced an explicit geometric construction of certain $4$-dimensional Lagrangian fibered hyper-K\"ahler orbifolds, and realized their dual Lagrangian fibrations using the same construction. It would be interesting to find a connection between their results and our perspective. Finally, \cite{ver99:mirror} discussed certain self-dualities of hyper-K\"ahler manifolds at the level of cohomology.
	\end{remark}
	
	There are two key ingredients for our proof of \Cref{thm:introduction}: the group $\Aut^{\circ}(X/B)$ and the notion of a polarization type. The definition of the group $\Aut^{\circ}(X/B)$ is inspired by the similar group $\Aut^{\circ}(X)$, which has already played an important role in the theory of hyper-K\"ahler manifolds. The two main properties of $\Aut^{\circ} (X)$ are its finiteness \cite{huy99} and deformation invariance \cite{has-tsc13}. The group $\Aut^{\circ}(X)$ is also computed for all known deformation types of hyper-K\"ahler manifolds (see \cite{bea83b}, \cite{boi-nei-sar11} and \cite{mon-wan17}). We provide similar results for the group $\Aut^{\circ}(X/B)$: it is finite abelian (\Cref{prop:aut0 is abelian}) and deformation invariant (\Cref{thm:aut0 is deformation invariant}). We also compute $\Aut^{\circ}(X/B)$ for all known deformation types in \Cref{thm:aut0 computation}. The idea of considering the polarization type of the fibers of $\pi_0$ has long been used, but only recently comprehensively studied by \cite{wie16, wie18}. We relate the polarization type to the study of our group scheme $K_0$.

	\subsection{Structure of the paper}
	In \Cref{sec:aut0 is deformation invariant}, we prove the group $\Aut^{\circ} (X/B)$ is deformation invariant on the Lagrangian fibration $\pi$. This is inspired by Hassett--Tschinkel's proof of deformation invariance of $\Aut^{\circ}(X)$ in \cite[Thm 2.1]{has-tsc13}. In \Cref{sec:abelian scheme}, we start by attaching an abelian scheme $P_0$ to any Lagrangian fibration of a hyper-K\"ahler manifold: $P_0$ is the identity component of the relative automorphism scheme of $\pi$. There exists a unique primitive polarization $\lambda$ on $P_0$ so that we can define its kernel group scheme $K_0$. We then try to relate $K_0$ and $\Aut^{\circ}(X/B)$ in general. This section also discusses the notion of the polarization type of a Lagrangian fibration. In essence, the polarization type is the study of a single fiber of the group scheme $K_0$.
	
	The goal of \Cref{sec:aut0 computation} is twofold. First, we compute the group $\Aut^{\circ} (X/B)$ for all currently known deformation types of hyper-K\"ahler manifolds. Second, we prove an inclusion $K_0 \subset \Aut^{\circ}(X/B)$ for special constructions of $\Kum_n$-type hyper-K\"ahler manifolds. The material here will be mostly concrete computations.
	\Cref{sec:minimal split covering} introduces a slightly more systematic method to assist this computations. In \Cref{sec:dual Lagrangian fibration}, we prove the main result of this article: there exists a natural compactification of the dual torus fibration for all currently known deformation types of hyper-K\"ahler manifolds. In \Cref{sec:example dual Kum2}, we give an illustration of the geometry and cohomology of $\check X$ when $X$ is of $\Kum_2$-type.
	
	We provide two appendices. \Cref{sec:singular HK} contains various definitions of singular hyper-K\"ahler varieties appearing in the literature. In \Cref{sec:quotient}, we discuss certain special quotients of compact hyper-K\"ahler manifolds. The quotient $\check X = X / K$ will be a special instance of this more general set-up.

	\subsection{Notation and conventions} \label{subsec:notation and conventions}
	In this paper, every hyper-K\"ahler manifold $X$ will be assumed to be compact but not necessarily projective unless stated explicitly. When $X$ further admits a Lagrangian fibration $\pi : X \to B$, it is helpful to keep in mind that $X$ is projective if and only if $\pi$ admits at least one rational multisection. Indeed, if $X$ is projective then a general scheme-theoretic fact says any smooth morphism between algebraic varieties admits an \'etale local section. The converse is \cite[Lem 2]{sawon09}.
	
	Assume $X$ has dimension $2n$. Any Lagrangian fibration $\pi : X \to B$ in this paper will always have the base $B = \PP^n$ since we are assuming $B$ is smooth and $0 < \dim B < 2n$ (see \cite{hwang08} and \cite{greb-lehn14}). The \emph{Beauville--Bogomolov form} and the \emph{Fujiki constant} of $X$ are a unique primitive symmetric bilinear form $q : H^2 (X, \ZZ) \otimes H^2 (X, \ZZ) \to \ZZ$ and a positive rational number $c_X$ satisfying the \emph{Fujiki relation}
	\begin{equation} \label{eq:fujiki relation}
		\int_X x^{2n} = c_X \cdot \frac{(2n)!}{2^n \cdot n!} \cdot q(x)^n \qquad\mbox{for}\quad x \in H^2 (X, \ZZ) .
	\end{equation}
	The Fujiki constant is computed for all currently known deformation types of hyper-K\"ahler manifolds: (1) $c_X = 1$ for $\text{K3}^{[n]}$ or OG10-type, and (2) $c_X = n+1$ for $\Kum_n$ or OG6-type (see \cite{bea83} and \cite{rap07, rap08}). In practice, we will mostly need a stronger version of the Fujiki relation, which follows from the polarization process
	\[ \int_X x_1 \cdots x_{2n} = c_X \sum_{\sigma} q(x_{\sigma(1)}, x_{\sigma(2)}) \cdots q(x_{\sigma(2n-1)}, x_{\sigma(2n)}) \qquad\mbox{for}\quad x_i \in H^2 (X, \ZZ) .\]
	Here $\sigma \in \mathfrak S_{2n}$ runs through all the $2n$-permutations but up to $2^n \cdot n!$ ambiguities inducing the same expression in the summation. The \emph{divisibility} of $x \in H^2 (X, \ZZ)$ is defined to be a positive integer
	\begin{equation} \label{eq:divisibility}
		\operatorname{div}(x) = \gcd \{ q(x,y) : y \in H^2 (X, \ZZ) \} .
	\end{equation}
	The study of the full cohomology $H^* (X, \QQ)$ will need the notion of the \emph{LLV algebra} $\mathfrak g$, introduced by Looijenga--Lunts \cite{loo-lunts97} and Verbitsky \cite{ver95}. For its concrete computations we will follow the representation theoretic notation used in \cite[\S 2--3]{gklr22}.
	
	Throughout, group schemes will be used both in algebraic and analytic context. A \emph{group scheme} is a morphism $G \to S$ equipped with an identity section $S \to G$, a group law morphism $G \times_S G \to G$ and an inverse $G \to G$ satisfying the usual axioms (either in the algebraic or analytic setting). An \emph{abelian scheme} $P \to S$ is an analytically proper connected commutative group scheme over $S$ with complex torus fibers. Any abelian scheme $P$ admits a \emph{dual abelian scheme} $\check P$. A \emph{polarization} of an abelian scheme $P$ is a finite \'etale homomorphism $\lambda : P \to \check P$ over $S$ such that for each fiber $F$, the restriction $\lambda_{|F} : F \to \check F$ is of the form $x \mapsto [t_x^* L \otimes L^{-1}]$ for an ample line bundle $L$ on $F$. Given a group scheme $G \to S$, an \emph{analytic torsor under $G$} (or \emph{analytic $G$-torsor}) is a morphism $Y \to S$ equipped with a $G$-action, such that there exists an analytic covering $\tilde S = \bigsqcup_{\alpha} U_{\alpha} \to S$ where the base change $\tilde Y = Y \times_S \tilde S$ and $\tilde G = G \times_S \tilde S$ are $\tilde G$-equivariantly isomorphic over $\tilde S$. In the algebraic setting, one can use a different topology, e.g., \'etale topology to define an \'etale torsor. Our reference for the theory of abelian schemes is \cite{mum:git}, \cite{neron} and \cite{fal-chai:abelian}. For the notion of torsors, see \cite{milne:etale} or \cite{neron}.

	\subsection{Acknowledgment}
	I would like to thank my advisor Radu Laza for his invaluable comments and suggestions. Without his guidance, I wouldn't have completed this paper. I would like to thank Igor Dolgachev, Mark Gross, Brendan Hassett, Klaus Hulek, Daniel Huybrechts, Olivier Martin, Hyeonjun Park and Qizheng Yin for helpful suggestions to improve the manuscript of this paper. Several people pointed out errors in the original version of this paper. I would like to thank Thorsten Beckmann, Salvatore Floccari, Mirko Mauri and Jieao Song for letting me know various errors in my initial arguments. I would especially like to thank Salvatore Floccari for pointing out a gap in the original version of the main statement. This work was partially supported by NSF Grant DMS-1802128 (PI Laza).

	\section{Deformation invariance of the $H^2$-trivial automorphisms} \label{sec:aut0 is deformation invariant}
	Let $X$ be a compact hyper-K\"ahler manifold. Consider the group of \emph{$H^2$-trivial automorphisms}
	\[ \Aut^{\circ} (X) = \ker \big( \Aut(X) \to \operatorname{O} (H^2 (X, \ZZ), q), \quad f \mapsto f_* \big) .\]
	Here $\Aut(X)$ is the group of biholomorphic automorphisms of $X$. Huybrechts \cite[Prop 9.1]{huy99} together with Hassett--Tschinkel \cite[Thm 2.1]{has-tsc13} proved that $\Aut^{\circ} (X)$ is a finite group which is invariant under deformations of $X$.
	
	Let us now further assume $X$ admits a Lagrangian fibration $\pi : X \to B$. We can restrict our attention to $H^2$-trivial automorphisms that respect the Lagrangian fibration
	\begin{equation}
		\Aut^{\circ} (X/B) = \Aut(X/B) \cap \Aut^{\circ}(X) .
	\end{equation}
	Since $\Aut^{\circ}(X)$ is finite, so is $\Aut^{\circ}(X/B)$. In fact, we can further prove $\Aut^{\circ}(X/B)$ is abelian: this will be showed later in \Cref{prop:aut0 is abelian}. Notice that $\Aut^{\circ}(X/B)$ not only depends on $X$ but also on the Lagrangian fibration $\pi : X \to B$. Hence, if $X$ admits two different Lagrangian fibrations then they may have different $\Aut^{\circ}(X/B)$. In \Cref{sec:aut0 computation}, we will compute $\Aut^{\circ} (X/B)$ for all currently known deformation types of hyper-K\"ahler manifolds $X$. In \Cref{sec:abelian scheme}, we will reinterpret $\Aut^{\circ} (X/B)$ as global sections of the ``translation automorphism scheme'' $P_0 \to B_0$.
	
	But before doing so, here we establish a more basic fact in this section; we prove $\Aut^{\circ} (X/B)$ is deformation invariant on $\pi$. To make this more precise, we first need to define the notion of a family of Lagrangian fibered hyper-K\"ahler manifolds.
	
	\begin{definition} \label{def:family of Lagrangian fibration}
		A \emph{family of Lagrangian fibered compact hyper-K\"ahler manifolds} is a commutative diagram
		\[\begin{tikzcd}[row sep=tiny]
			\mathcal X \arrow[rd, "\pi"] \arrow[dd, "p"] \\
			& \mathcal B \arrow[ld, "q"] \\
			S
		\end{tikzcd}\]
		with the following conditions.
		\begin{enumerate}
			\item $p : \mathcal X \to S$ is a smooth proper family of compact hyper-K\"ahler manifolds of dimension $2n$ over a complex analytic space $S$.
			\item $q : \mathcal B \to S$ is the projectivization of a rank $n+1$ holomorphic vector bundle on $S$.
			\item For all $t \in S$, the fiber $\pi : X_t \to B_t$ is a Lagrangian fibration.
		\end{enumerate}
	\end{definition}
	
	Note that the second condition ensures $\mathcal B$ is projective over $S$ and admits a relative ample line bundle $\mathcal O_{\mathcal B/S}(1)$. It is also possible to consider a weaker version of this definition which only assumes $\mathcal B \to S$ to be a $\PP^n$-bundle. The obstruction for a $\PP^n$-bundle to be the projectivization of a vector bundle lies in the analytic Brauer group $H^2 (S, \mathcal O_S^*)$. Thus, if $H^2 (S, \mathcal O_S^*) = 0$ (for example, when $S$ is a complex open ball) then the second axiom is in fact equivalent to the weaker one. Notice that the pullback $\mathcal H = \pi^* \mathcal O_{\mathcal B/S}(1)$ can be considered as a family of line bundles $H_t$ on $X_t$. Therefore, \Cref{def:family of Lagrangian fibration} induces a family of pairs $(X, H)$ where $H = \pi^* \mathcal O_B(1)$.
	
	As usual, two Lagrangian fibrations $\pi : X \to B$ and $\pi' : X' \to B'$ are \emph{deformation equivalent} if there exists a family of Lagrangian fibered compact hyper-K\"ahler manifolds $\mathcal X / \mathcal B / S$ over a connected union of $1$-dimensional open disks $S$, realizing them as two fibers at $t, t' \in S$. Matsushita in \cite{mat16} proved such a deformation problem admits a local universal deformation.
	
	We can now state the main theorem of this section.
	
	\begin{theorem} \label{thm:aut0 is deformation invariant}
		The group $\Aut^{\circ}(X/B)$ is invariant under deformations of $\pi : X \to B$.
	\end{theorem}
	
	The rest of this section will be devoted to the proof of \Cref{thm:aut0 is deformation invariant}. The sketch of the proof is as follows. First, we descend the $\Aut^{\circ}(X)$-action on $X$ to $B$ so that the Lagrangian fibration $\pi : X \to B$ becomes an equivariant morphism. This means we have a group homomorphism $\Aut^{\circ} (X) \to \Aut(B)$ whose kernel is precisely $\Aut^{\circ}(X/B)$. Descending such an action is a nontrivial problem (this is quite similar to the result of \cite{brion11}), so we need to overcome this issue using the notion of a $G$-linearizability of line bundles. Next, we need to sheafify the discussions as we are interested in the deformation behavior of them. The result will follow from formal properties of the kernel of the sheaf homomorphism.

	\subsection{$G$-linearizability of a line bundle}
	Before we get into the proof of \Cref{thm:aut0 is deformation invariant}, let us recall the notion of $G$-linearizability of a line bundle on a complex manifold. For simplicity we only consider finite group actions. Our references are \cite[\S 3]{brion18}, \cite[\S 7]{dol:inv} and \cite{mum:git}, but we need to take some additional care since these references only consider the algebraic setting.
	
	Let $G$ be an arbitrary finite group and $\mathcal X$ be a complex manifold equipped with a holomorphic $G$-action. A \emph{$G$-linearized line bundle} on $\mathcal X$ is a holomorphic line bundle $\mathcal L$ together with a collection of isomorphisms $\Phi_g : g^* \mathcal L \to \mathcal L$ for $g \in G$, satisfying the condition $\Phi_{gg'} = \Phi_{g'} \circ g'^* \Phi_g$ for $g, g' \in G$. A \emph{$G$-invariant line bundle} on $\mathcal X$ is a holomorphic line bundle $\mathcal L$ such that $g^* \mathcal L \cong \mathcal L$ for all $g \in G$ (without any condition). We denote by $\Pic^G(\mathcal X)$ and $\Pic(\mathcal X)^G$ the groups of $G$-linearized line bundles and $G$-invariant line bundles on $\mathcal X$ up to isomorphisms. The second group is precisely the $G$-invariant subgroup of $\Pic(\mathcal X)$.
	
	There is a forgetful homomorphism $\Pic^G (\mathcal X) \to \Pic (\mathcal X)^G$, which is neither injective nor surjective in general. To understand the obstruction to its surjectivity, one considers an exact sequence of abelian groups (\cite[Rmk 7.2]{dol:inv} or \cite[Prop 3.4.5]{brion18})
	\[\begin{tikzcd}
		\Pic^G(\mathcal X) \arrow[r] & \Pic(\mathcal X)^G \arrow[r] & H^2 (G, \Gamma)
	\end{tikzcd}, \qquad \Gamma = H^0 (\mathcal X, \mathcal O_{\mathcal X}^*) .\]
	Both Dolgachev and Brion's discussions are for algebraic varieties, but their proofs can be adapted to our analytic setting as well. With this exact sequence in hand, we have:
	
	\begin{lemma} \label{lem:linerizable}
		Every $G$-invariant line bundle $\mathcal H$ on $\mathcal X$ is $G$-linerizable up to a suitable tensor power.
	\end{lemma}
	\begin{proof}
		It is a general fact in the theory of group cohomology (for finite groups) that all the higher degree cohomologies $H^{\ge 1} (G, \Gamma)$ are $|G|$-torsion for any $G$-module $\Gamma$ (e.g., \cite[Cor VIII.1]{serre:local_fields}). Hence by the exact sequence above, the $|G|$-th tensor $\mathcal H^{\otimes |G|}$ vanishes in $H^2 (G, \Gamma)$ and hence comes from $\Pic^G (\mathcal X)$.
	\end{proof}
	
	For us, the importance of the $G$-linearizability of a line bundle comes from the induced $G$-action on the higher direct images of a linearized line bundle. If $\mathcal L$ is a $G$-linearized line bundle on $\mathcal X$ and $p : \mathcal X \to S$ is a $G$-invariant holomorphic map, then we have a contravariant $G$-action on all the higher direct image sheaves
	\[ g^* : R^k p_* \mathcal L \to R^k p_* \mathcal L, \qquad (g \circ g')^* = g'^* \circ g^* .\]
	Now assume further $\mathcal L$ is globally generated over $S$ and $p_* \mathcal L$ is a vector bundle on $S$. Then we have a $G$-action on $\PP_S (p_* \mathcal L)$ making the holomorphic map $\mathcal X \to \PP_S (p_* \mathcal L)$ $G$-equivariant over $S$. See \cite[Prop 1.7]{mum:git}.

	\subsection{The automorphism sheaves and deformation invariance of the $H^2$-trivial automorphisms}
	Suppose we have a smooth proper family of hyper-K\"ahler manifolds $p : \mathcal X \to S$. Let $U \subset S$ be an analytic open subset and denote by $p : \mathcal X_U = p^{-1} (U) \to U$ the restricted family over $U$. We define the sheaf of $H^2$-trivial automorphism groups $\SheafAut^{\circ}_{\mathcal X/S}$ on $S$ by
	\[ \SheafAut^{\circ}_{\mathcal X/S} (U) = \{ f : \mathcal X_U \to \mathcal X_U : U \mbox{-automorphism such that } f^* : R^2 p_* \ZZ \to R^2 p_* \ZZ \mbox{ is the identity} \} .\]
	By the work of Huybrechts and Hassett--Tschinkel, this sheaf is a local system of finite groups. We can consider it as a family of groups $\Aut^{\circ} (X_t)$ for $t \in S$. Similarly, given a family of Lagrangian fibered hyper-K\"ahler manifolds, we can define a family of groups $\Aut^{\circ}(X_t/B_t)$:
	
	\begin{definition}
		Given a family of Lagrangian fibered hyper-K\"ahler manifolds $p : \mathcal X \xrightarrow{\pi} \mathcal B \xrightarrow{q} S$, we define a sheaf of groups $\SheafAut^{\circ}_{\mathcal X/ \mathcal B/S}$ on $S$ by
		\[ \SheafAut^{\circ}_{\mathcal X / \mathcal B/S} (U) = \{ f : \mathcal X_U \to \mathcal X_U : \mathcal B_U \mbox{-automorphism such that } f^* : R^2 p_* \ZZ \to R^2 p_* \ZZ \mbox{ is the identity} \} .\]
		Equivalently, we may define $\SheafAut^{\circ}_{\mathcal X/\mathcal B/S} = q_* \SheafAut_{\mathcal X/\mathcal B} \cap \SheafAut^{\circ}_{\mathcal X/S}$.
	\end{definition}
	
	As mentioned, the sheaf of $H^2$-trivial automorphisms $\SheafAut^{\circ}_{\mathcal X/S}$ is a local system. The sheaf $\SheafAut^{\circ}_{\mathcal X / \mathcal B / S}$ is a subsheaf of $\SheafAut^{\circ}_{\mathcal X/S}$, and our goal is to prove it is locally constant as well. The question is certainly local on the base $S$, so we may assume $S$ is a small open ball. Then $\SheafAut^{\circ}_{\mathcal X/S}$ becomes a constant sheaf, so we may consider it as an abstract finite group
	\[ G = \Aut^{\circ} (X) \]
	acting on $\mathcal X \to S$ fiberwise.
	
	Consider the automorphism sheaf $\SheafAut_{\mathcal B/S}$ of the $\PP^n$-bundle $\mathcal B \to S$. It is the sheaf of analytic local sections of the $\PGL(n+1, \CC)$-group scheme $\Aut_{\mathcal B/S} \to S$. Our first step is to realize the sheaf $\SheafAut^{\circ}_{\mathcal X/\mathcal B/S}$ as the kernel of a certain homomorphism $\SheafAut^{\circ}_{\mathcal X/S} \to \SheafAut_{\mathcal B/S}$.
	
	\begin{proposition} \label{prop:kernel of equivariant homomorphism}
		Assume $S$ is an open ball. Then there exists a homomorphism of sheaves
		\begin{equation} \label{eq:homomorphism of automorphism sheaves}
			G = \SheafAut^{\circ}_{\mathcal X/S} \to \SheafAut_{\mathcal B/S}
		\end{equation}
		whose kernel is $\SheafAut^{\circ}_{\mathcal X/\mathcal B/S}$.
	\end{proposition}
	
	Equivalently, the proposition states that there exists a $G$-action on $\mathcal B$ making $\pi : \mathcal X \to \mathcal B$ a $G$-equivariant morphism over $S$. To prove the proposition, we need to use the $G$-linearizability of line bundles in the previous subsection. The following lemma proves every line bundle on $\mathcal X$ is $G$-invariant.
	
	\begin{lemma}
		$G$ acts trivially on $\Pic(\mathcal X)$.
	\end{lemma}
	\begin{proof}
		We first claim $G$ acts trivially on $H^2 (\mathcal X, \ZZ)$. Apply the Leray spectral sequence
		\[ E_2^{p,q} = H^p (S, R^q p_* \ZZ) \quad \Rightarrow \quad H^{p+q} (\mathcal X, \ZZ) .\]
		Noticing that $R^0 p_* \ZZ = \ZZ$, $R^1 p_* \ZZ = 0$, and $S$ is an open ball, we obtain an isomorphism $H^2 (\mathcal X, \ZZ) \cong H^0 (S, R^2 p_*\ZZ)$. This isomorphism respects the $G$-action as the Leray spectral sequence is functorial. Now $G$ acts on $H^2 (X_t, \ZZ)$ trivially for any fiber $X_t$, so $G$ acts on $R^2p_*\ZZ$ trivially and the claim follows.
		
		It is enough to prove the first Chern class map $\Pic (\mathcal X) \to H^2 (\mathcal X, \ZZ)$ is injective. This homomorphism is induced by the exponential sequence $0 \to \ZZ \to \mathcal O_{\mathcal X} \to \mathcal O_{\mathcal X}^* \to 0$, so it suffices to prove $H^1(\mathcal X, \mathcal O_{\mathcal X}) = 0$. Again, use the Leray spectral sequence
		\[ E_2^{p, q} = H^p (S, R^q p_* \mathcal O_{\mathcal X}) \quad \Rightarrow \quad H^{p+q} (\mathcal X, \mathcal O_{\mathcal X}) .\]
		This time, we have $R^0 p_* \mathcal O_{\mathcal X} = \mathcal O_S$ and $R^1 p_* \mathcal O_{\mathcal X} = 0$. This implies $H^1(\mathcal X, \mathcal O_{\mathcal X}) = 0$.
	\end{proof}
	
	Consider the line bundle $\mathcal H = \pi^* \mathcal O_{\mathcal B/S}(1)$ on $\mathcal X$. Since $\Pic(\mathcal X)$ is $G$-invariant, we can apply \Cref{lem:linerizable} to $\mathcal H$ and conclude $\mathcal H^{\otimes m}$ is $G$-linearizable for some positive integer $m$. As a result, we have a $G$-equivariant morphism $\pi_m : \mathcal X \to \mathcal B_m$ where $\mathcal B_m = \PP_S \big( p_* \mathcal H^{\otimes m} \big)$ is the dual of the complete linear system associated to $\mathcal H^{\otimes m}$. Consider the diagram
	\begin{equation} \label{diag:iitaka}
	\begin{tikzcd}[row sep=tiny]
		\mathcal X \arrow[dd, "p"] \arrow[rd, "\pi"] \arrow[rrd, bend left=20, "\pi_m"] \\
		& \mathcal B \arrow[r, hookrightarrow] \arrow[ld, "q"] & \mathcal B_m \arrow[lld, bend left=20, "q_m"] \\
		S
	\end{tikzcd}.
	\end{equation}
	
	\begin{lemma}
		The $m$-th relative Veronese embedding $\mathcal B \hookrightarrow \mathcal B_m$ makes the diagram \eqref{diag:iitaka} commute.
	\end{lemma}
	\begin{proof}
		Notice that the morphism $\pi : \mathcal X \to \mathcal B$ associated to $\mathcal H$ has connected fibers. Hence, for any $t \in S$, the fiber $\pi : X_t \to B_t$ becomes the Iitaka fibration of the line bundle $H_t$ (e.g., \cite[\S 2.1.B]{laz1}). This in particular implies that any morphism $\pi_m : X_t \to (B_m)_t$ associated to $H_t^{\otimes m}$ factors through the Iitaka fibration $\pi$, where the morphism $B_t \hookrightarrow (B_m)_t$ is precisely the $m$-th Veronese embedding. In other words, the $m$-th relative Veronese embedding makes the diagram commute.
	\end{proof}
	
	The equivariance of $\pi_m$ with the diagram \eqref{diag:iitaka} implies $\pi$ is equivariant, completing the proof of \Cref{prop:kernel of equivariant homomorphism}. We now present the proof of the main theorem.
	
	\begin{proof} [Proof of \Cref{thm:aut0 is deformation invariant}]
		Let $\mathcal X \to \mathcal B \to S$ be a family of Lagrangian fibered hyper-K\"ahler manifolds over an open ball $S$. The sheaves $\SheafAut^{\circ}_{\mathcal X/S}$ and $\SheafAut_{\mathcal B/S}$ are represented by the constant group schemes
		\[ \Aut^{\circ}_{\mathcal X/S} \cong \bigsqcup_{f \in \Aut^{\circ}(X)} S , \qquad \Aut_{\mathcal B/S} \cong \PGL(n+1) \times S .\]
		The homomorphism \eqref{eq:homomorphism of automorphism sheaves} becomes a morphism $\alpha : \Aut^{\circ}_{\mathcal X/S} \to \Aut_{\mathcal B/\mathcal S}$. The desired sheaf $\SheafAut^{\circ}_{\mathcal X/\mathcal B/S}$ is representable by $\ker \alpha$ by \Cref{prop:kernel of equivariant homomorphism}.
		
		To prove $\ker \alpha$ is a constant subgroup scheme, it is enough to show the following: let $S'$ be a connected component of $\Aut^{\circ}_{\mathcal X/S}$. Consider the restriction of $\alpha$ followed by the projection
		\[ \beta : S' \to \PGL(n+1) .\]
		Then we claim that either $\beta(S') = \{ \id \}$ or $\beta(S') \not \ni \id$. Notice that the image $\beta(S')$ consists of $|G|$-torsion matrices in $\PGL(n+1)$. Since the set of $|G|$-torsion matrices is a disjoint union of $\PGL(n+1)$-adjoint orbits (classified by eigenvalues), the connected set $\beta(S')$ has to lie in a single orbit. The adjoint orbit containing the identity matrix is a singleton set $\{ \id \}$. Hence the claim follows.
	\end{proof}

	\section{Abelian schemes associated to Lagrangian fibrations} \label{sec:abelian scheme}
	The aim of this section is to associate a polarized abelian scheme to \emph{every} Lagrangian fibered compact hyper-K\"ahler manifold, and to discuss its consequences. The following is the first main theorem of this section.
	
	\begin{theorem} \label{thm:abelian scheme}
		Let $\pi : X \to B$ be a Lagrangian fibration of a compact hyper-K\"ahler manifold and $B_0 \subset B$ its smooth locus. Set $X_0 = \pi^{-1} (B_0)$ so that it becomes a smooth proper family of complex tori over $B_0$.
		\begin{enumerate}
			\item There exists a unique projective abelian scheme $\nu : P_0 \to B_0$ making $\pi : X_0 \to B_0$ an analytic torsor under $\nu$.
			\item Moreover, the abelian scheme is simple and has a unique choice of a primitive polarization
			\begin{equation} \label{eq:polarization}
				\lambda : P_0 \to \check P_0 .
			\end{equation}
			Here $\check P_0 \to B_0$ is the dual abelian scheme of $P_0 \to B_0$.
		\end{enumerate}
	\end{theorem}
	
	\begin{definition}
		The abelian scheme $\nu : P_0 \to B_0$ in \Cref{thm:abelian scheme} is called the \emph{abelian scheme associated to $\pi$}.
	\end{definition}
	
	Our statement is motivated by Arinkin--Fedorov's result in \cite[Thm 2]{ari-fed16}, van Geemen--Voisin's argument in \cite{vgee-voi16}, and Sawon's result in \cite{sawon04}. The theorem combines and slightly generalizes these results. Before discussing the applications of this theorem, let us first present some examples.
	
	\begin{example}
		Let $X$ be a smooth projective moduli of torsion coherent sheaves on a K3 surface with a fixed Mukai vector, so that it becomes a hyper-K\"ahler manifold of $\text{K3}^{[n]}$-type equipped with a Lagrangian fibration $\pi : X \to B$ (see, e.g., \cite{decat-rap-sac21}). In this case, it is known that the torus fibration $\pi : X_0 \to B_0$ is isomorphic to a relative Jacobian $\Pic^d_{\mathcal C/B_0}$ associated to a certain universal family $\mathcal C/B_0$ of smooth curves on the K3 surface. Now $\Pic^d_{\mathcal C/B_0}$ is a torsor under the numerically trivial relative Jacobian $\Pic^0_{\mathcal C/B_0}$ \cite[Thm 9.3.1]{neron}. By the uniqueness assertion of \Cref{thm:abelian scheme}, this is the associated abelian scheme $P_0$.
	\end{example}
	
	\begin{example}
		When $\pi : X \to B = \PP^1$ is an elliptic K3 surface, \Cref{thm:abelian scheme} is a weaker version of the relative Jacobian fibration construction of $\pi$ (e.g., \cite[\S 11.4]{huy:k3}). In this case, one may even construct a semi-abelian scheme $P \to B$ over the entire base (N\'eron model) so that the smooth locus of $\pi$ becomes a torsor under $P$. Arinkin--Fedorov generalized this result to certain higher dimensional projective hyper-K\"ahler manifolds. A stronger version of \Cref{thm:abelian scheme} would potentially improve the arguments in this paper, but we will not discuss this further.
	\end{example}
	
	\begin{example}
		For higher dimensional compact hyper-K\"ahler manifolds, one may still consider the relative Picard scheme $\Pic^0_{X_0/B_0} \to B_0$. However, in general this is the \emph{dual} of the abelian scheme $P_0$. This means we can consider $P_0$ as the ``double Picard scheme'' of the original $X_0$ \cite{sawon04}. However, for us it will be more useful to consider $P_0$ as the identity component of the relative automorphism scheme of $X_0/B_0$. We will show this in \Cref{prop:abelian scheme is automorphism scheme}.
	\end{example}
	
	\begin{example} \label{ex:when pi has a rational section}
		When $\pi$ admits at least one rational section, then the abelian scheme $P_0$ is in fact isomorphic to $X_0$. This is because the rational section must be defined over $B_0$ by \Cref{rmk:rational section is defined over B0}, so that $X_0$ becomes a trivial $P_0$-torsor. In some sense, \Cref{thm:abelian scheme} is thus a generalization of certain properties of $X_0$ to the case where $\pi$ does not have any rational section. For example, one can study the Mordell--Weil group of $\nu$, generalizing the study of the Mordell--Weil group of $\pi$.
	\end{example}
	
	One application of \Cref{thm:abelian scheme} is a more systematic study of the polarization type of the torus fibers arising in $\pi$. The study of the polarization type of the torus fibers goes back to at least \cite{sawon03}, which in turn references an earlier idea of Mukai (see Proposition 5.3 in loc. cit.). However, to our knowledge, Wieneck's series of papers \cite{wie16, wie18} were the first work to consider the polarization type as an invariant attached to a Lagrangian fibration and study them in great details for $\text{K3}^{[n]}$ and $\Kum_n$-type hyper-K\"ahler manifolds. Using \Cref{thm:abelian scheme}, we can given an alternative definition of the polarization type.
	
	\begin{definition} \label{def:polarization scheme}
		\begin{enumerate}
			\item The \emph{polarization scheme} of $\pi$ is the kernel
			\[ K_0 = \ker \lambda \]
			of the polarization \eqref{eq:polarization}.
			
			\item The \emph{polarization type} of $\pi$ is an $n$-tuple of positive integers $(d_1, \cdots, d_n)$ with $d_1 \mid \cdots \mid d_n$ such that the fibers of the polarization scheme are isomorphic to $(\ZZ/d_1 \oplus \cdots \oplus \ZZ/d_n)^{\oplus 2}$.
		\end{enumerate}
	\end{definition}
	
	The polarization scheme $K_0$ is a finite \'etale commutative group scheme over $B_0$. Hence its fibers are all isomorphic and the polarization type is well-defined. The polarization type will be an important ingredient for our method. We devote a short subsection \ref{subsec:polarization type} to collect its properties.
	
	The second theme of this section is a relation between the group $\Aut^{\circ}(X/B)$ and the polarization scheme $K_0$. We will see in \Cref{prop:aut defines global section of P} that every automorphism $f \in \Aut^{\circ} (X/B)$ defines a global section of the abelian scheme $P_0 \to B_0$. If we consider $\Aut^{\circ}(X/B)$ as a constant group scheme over $B_0$, this means we have a closed immersion of group schemes
	\begin{equation} \label{eq:aut defines global section of P}
		\Aut^{\circ} (X/B) \hookrightarrow P_0 .
	\end{equation}
	We expect the image of this injective map will \emph{contain} the polarization scheme $K_0$. This is a nontrivial claim; this would imply the polarization scheme is a constant group scheme and is extendable to $K \to B$ acting on the entire $X \to B$. We were not able to prove this claim in general, and a large part of this paper will be devoted to showing this for known deformation types of hyper-K\"ahler manifolds. The following propositions will be our technical tools for doing this. It will be convenient to introduce a temporary notation
	\[ K_0[a] = \ker (a \lambda : P_0 \to \check P_0) ,\]
	a finite \'etale commutative group scheme over $B_0$.
	
	\begin{proposition} \label{prop:polarization scheme ver1}
		Let $\pi : X \to B$ and $\pi' : X' \to B'$ be two deformation equivalent Lagrangian fibrations of compact hyper-K\"ahler manifolds. Let $a$ be any positive integer. Then the inclusion \eqref{eq:aut defines global section of P} factors through
		\begin{equation} \label{eq:aut defines global section of aK}
			\Aut^{\circ} (X/B) \hookrightarrow K_0[a]
		\end{equation}
		if and only if the same holds for $\pi'$.
	\end{proposition}
	
	\begin{proposition} \label{prop:polarization scheme ver2}
		Let $\pi : X \to B$ be a Lagrangian fibration of a compact hyper-K\"ahler manifold and $(d_1, \cdots, d_n)$ its polarization type. Assume we have an equality $c_X = d_1 \cdots d_n$. Then \eqref{eq:aut defines global section of aK} holds for $a = \operatorname{div}(h)$, where $h \in H^2 (X, \ZZ)$ is the class of $\pi^* \mathcal O_B(1)$ and its divisibility $\operatorname{div}(h)$ is as defined in \eqref{eq:divisibility}.
	\end{proposition}
	
	Note that the inclusion \eqref{eq:aut defines global section of aK} has a different direction from our desired $K_0 \hookrightarrow \Aut^{\circ}(X/B)$. Our strategy will be to first show \eqref{eq:aut defines global section of aK} for a certain value of $a$, and then deduce the relation between two subgroup schemes $K_0, \Aut^{\circ}(X/B) \subset K_0[a]$. The first proposition says the inclusion \eqref{eq:aut defines global section of aK} is deformation invariant on $\pi$. The second proposition provides at least one such an integer $a$, though this may not be the minimum possible value. The unfortunate assumption $c_X = d_1 \cdots d_n$ will be satisfied for all known deformation types of hyper-K\"ahler manifolds, so it will not be a huge problem. See \Cref{thm:polarization type computation}.
	
	Finally, we would like to mention a special consequence of the above discussion when $\pi$ admits a rational section. This may give the readers some more ideas on these objects. Recall from \Cref{ex:when pi has a rational section} that the existence of a rational section implies $X_0 \cong P_0$. Therefore, \eqref{eq:aut defines global section of P} implies the existence of certain torsion rational sections of the Lagrangian fibration
	\[ \Aut^{\circ}(X/B) \subset \MW(X/B) .\]
	For example, let us consider the case when $X$ is of $\Kum_n$-type. We will prove in \Cref{thm:aut0 computation} that the order of $\Aut^{\circ}(X/B)$ is at least $(n+1)^2$. Thus any Lagrangian fibration of a $\Kum_n$-type hyper-K\"ahler manifold \emph{must} have at least $(n+1)^2$ torsion rational sections (once it admits a single torsion rational section), and the dual hyper-K\"ahler orbifold $\check X$ is precisely the quotient of $X$ by these special torsion rational sections. To our knowledge, this phenomenon has not been observed before and it became one of our original motivations. See also \cite[\S 3--5]{sacca20} for some related ideas on the Mordell--Weil group and birational automorphisms defined by torsion rational sections.

	\subsection{Abelian scheme associated to a Lagrangian fibration} \label{sec:polarization}
	In this subsection, we present the proof of \Cref{thm:abelian scheme}. Note again that we are assuming neither $X$ is projective nor $\pi$ has a rational section.
	
	Recall that every smooth closed fiber $F$ of $\pi$ is a complex torus (holomorphic Arnold--Liouville theorem). In fact, $F$ is necessarily an abelian variety as observed by Voisin \cite[Prop 2.1]{cam06}. It would be helpful for us to review this fact. The key idea is the following cohomological lemma, which has been discovered several times independently in \cite{voi92, ogu09picard, mat16} and recently generalized into higher degree cohomologies by Shen--Yin and Voisin \cite{shen-yin22}.
	
	\begin{lemma} \label{lem:matsushita}
		Let $F$ be any smooth fiber of $\pi$ and $h \in H^2 (X, \ZZ)$ the cohomology class of $\pi^* \mathcal O_B(1)$. Then the restriction map
		\[ -_{|F} : H^2 (X, \ZZ) \to H^2 (F, \ZZ) \]
		has $\ker (-_{|F}) = h^{\perp}$. Consequently, it has $\im (-_{|F}) \cong \ZZ$.
	\end{lemma}
	
	\begin{corollary} [Voisin] \label{cor:voisin}
		The image of the restriction map $-_{|F}$ is generated by an ample class of $F$. As a result, $F$ is an abelian variety.
	\end{corollary}
	\begin{proof}
		Say $y$ is an integral generator of \Cref{lem:matsushita}. Choose any K\"ahler class $\omega \in H^2 (X, \RR)$ and consider its restriction $\omega_{|F}$, a K\"ahler class on $F$. It has to be a nonzero real multiple of $y$. This means, up to sign, $y$ has to be a K\"ahler class on $F$. Hence $y$ is an integral K\"ahler class, so it is ample.
	\end{proof}
	
	We caution the reader to be aware that the ample generator $y$ of the image of the restriction map need not be primitive (see \Cref{prop:non primitive polarization}). One reasonable choice of a polarization on an abelian variety fiber $F$ is a unique \emph{primitive} ample class in $H^2 (F, \ZZ)$ parallel to $y$. \Cref{thm:abelian scheme} is essentially a more global way to formulate this over the whole base $B_0$.
	
	We divide the proof of \Cref{thm:abelian scheme} into three parts: (1) an explicit construction of the polarized abelian scheme $P_0$, (2) proving such a construction makes $X_0$ a torsor under $P_0$, and finally (3) its uniqueness. The uniqueness should be a more general fact about arbitrary torsors, at least in the algebraic case (see  Moret-Bailly's answer in \cite{mathoverflow:mor}). The construction of $P_0$ works for any proper family of complex tori. The uniqueness of the polarization is the only part that needs the fact $X_0$ is obtained from a Lagrangian fibered hyper-K\"ahler manifold $X$.
	
	The proof of the construction part closely follows \cite{vgee-voi16}, but for completeness we reproduce their argument here. 
	
	\begin{proof} [Proof of \Cref{thm:abelian scheme}, construction]
		Apply the global invariant cycle theorem (for proper maps between compact K\"ahler manifolds \cite{deligne71}) and \Cref{lem:matsushita} to obtain
		\[ H^0 (B_0, R^2 \pi_* \QQ) = \im (H^2 (X, \QQ) \to H^2 (F, \QQ)) \cong \QQ .\]
		Hence, there exists a unique homomorphism $(R^2 \pi_* \QQ)^{\vee} \to \QQ$ of local systems on $B_0$ up to scalar. This is a homomorphism of $\QQ$-VHS: fiberwise, \Cref{cor:voisin} proves the image of $H^2 (X, \QQ) \to H^2 (F, \QQ)$ is an ample class. Restrict it to the morphism of $\ZZ$-VHS $(R^2 \pi_* \ZZ)^{\vee} \to \ZZ$. The morphism can be uniquely determined once we assume it to be primitive and represents an ample class on each fiber. Finally, use the fact that $\pi : X_0 \to B_0$ is a family of complex tori (abelian varieties) and obtain an isomorphism $R^2 \pi_* \ZZ = \wedge^2 R^1 \pi_* \ZZ$. The result is a primitive polarization
		\begin{equation} \label{eq:polarization_vhs}
			(R^1 \pi_* \ZZ)^{\vee} \otimes (R^1 \pi_* \ZZ)^{\vee} \to \ZZ .
		\end{equation}
		
		We have constructed a weight $-1$ $\ZZ$-VHS $(R^1 \pi_* \ZZ)^{\vee}$ equipped with a polarization \eqref{eq:polarization_vhs}. Now use a formal equivalence of categories between polarized weight $-1$ $\ZZ$-VHS and that of polarized abelian schemes (e.g., \cite[\S 5.2]{deligne72} \cite[\S 4.4]{deligne71}). This constructs our desired abelian scheme $\nu : P_0 \to B_0$ with a unique primitive polarization $\lambda : P_0 \to \check P_0$ over $B_0$. To prove $P_0$ is simple, we may prove the corresponding VHS $R^1 \pi_* \QQ$ is simple. This is tacitly proved in \cite{vgee-voi16} and later explicitly stated in \cite[Lem 4.5]{voi18}. The idea is that if $R^1 \pi_* \QQ$ splits as a direct sum $\mathcal V_1 \oplus \mathcal V_2$ of two VHS, then each of them has their own polarizations, forcing $h^0 (B_0, R^2 \pi_* \QQ) \ge h^0 (B_0, \wedge^2 \mathcal V_1) + h^0 (B_0, \wedge^2 \mathcal V_2) \ge 2$. We omit the details here.
	\end{proof}
	
	\begin{proof} [Proof of \Cref{thm:abelian scheme}, torsor]
		Consider an analytic open covering $\{ B_i : i \in I \}$ of $B_0$ so that over each $B_i$, the restriction of the Lagrangian fibration $\pi : X_i \to B_i$ admits at least one holomorphic section $s_i : B_i \to X_i$. Considering $s_i$ as a zero section, $\pi : X_i \to B_i$ becomes an abelian scheme. Hence by the equivalence of abelian schemes and $(R^1\pi_*\ZZ)^{\vee}$, $\nu$ and $\pi$ are isomorphic over $B_i$ by $\phi_i : X_i \to P_i$ sending $s_i$ to the zero section of $P_i$.
		
		Now use the isomorphism $\phi_i$ to transform the group law $+ : P_i \times_{B_i} P_i \to P_i$ into a $P_i$-action on $X_i$. That is, we define a group action morphism by
		\[ \rho_i : P_i \times_{B_i} X_i \to X_i, \qquad (p_i, x_i) \mapsto \phi_i^{-1} (\phi_i (x_i) + p_i) .\]
		We want to patch $\rho_i$ together to define a group action $\rho : P_0 \times_{B_0} X_0 \to X_0$ over the entire $B_0$. To do so, we need to check whether the definitions of $\rho_i$ and $\rho_j$ coincides over the intersection $B_{ij} = B_i \cap B_j$, i.e.,
		\begin{equation} \label{eq:descent_condition}
			\phi_i^{-1} (\phi_i (x_{ij}) + p_{ij}) = \phi_j^{-1} (\phi_j (x_{ij}) + p_{ij}) \qquad \mbox{for all}\quad (p_{ij}, x_{ij}) \in P_{ij} \times_{B_{ij}} X_{ij} .
		\end{equation}
		
		Over $B_{ij}$, one has a transition function $\phi_j \circ \phi_i^{-1} : P_{ij} \to X_{ij} \to P_{ij}$, an automorphism of $P_{ij}$. Recall that the isomorphisms $\phi_i$ and $\phi_j$ are constructed by choosing the zero sections $s_i$ and $s_j$, and the corresponding isomorphisms $\phi_i : X_{ij} \cong P_{ij}$ and $\phi_j : X_{ij} \cong P_{ij}$ are as abelian schemes. From it, we notice the automorphism $\phi_j \circ \phi_i^{-1} : P_{ij} \to P_{ij}$ is a \emph{translation} automorphism. The translation is by $\phi_j \circ \phi_i^{-1}(0)$, the difference of the two zero sections. With this, we have a sequence of identities
		\begin{align*}
			\phi_j(x_{ij}) + p_{ij} = \phi_j \circ \phi_i^{-1} (\phi_i (x_{ij})) + p_{ij} &= \big( \phi_i (x_{ij}) + \phi_j \circ \phi_i^{-1} (0) \big) + p_{ij} \\
			&= \big( \phi_i (x_{ij}) + p_{ij} \big) + \phi_j \circ \phi_i^{-1} (0) = \phi_j \circ \phi_i^{-1} \big( \phi_i (x_{ij}) + p_{ij} \big) .
		\end{align*}
		This proves \eqref{eq:descent_condition}. Hence $\rho_i$ patches together and defines a morphism $\rho : P_0 \times_{B_0} X_0 \to X_0$. The group action axioms are all easily verified. Also, $X_0$ is clearly a $P_0$-torsor by construction.
	\end{proof}
	
	\begin{proof} [Proof of \Cref{thm:abelian scheme}, uniqueness]
		Let $\nu : P_0 \to B_0$ be a (not necessarily projective) abelian scheme so that $\pi$ becomes a torsor under $\nu$. We claim $R^1 \nu_* \ZZ \cong R^1 \pi_* \ZZ$ as VHS over $B_0$. Consider the group scheme action map
		\[\begin{tikzcd}
			P_0 \times_{B_0} X_0 \arrow[r, "\rho"] \arrow[dr, "\mu"'] & X_0 \arrow[d, "\pi"] \\
			& B_0
		\end{tikzcd}.\]
		From the diagram, we have a pullback morphism between the VHS $\rho^* : R^1 \pi_* \ZZ \to R^1 \mu_* \ZZ$. The latter VHS is isomorphic to the direct sum $R^1 \nu_* \ZZ \oplus R^1 \pi_* \ZZ$ by the K\"unneth formula (e.g., \cite[VII.2.7]{ive:sheaves}) and decomposition theorem for smooth proper morphisms. Hence composing with the first projection, we obtain a morphism $R^1 \pi_* \ZZ \to R^1 \nu_* \ZZ$. Now over a small analytic open subset $U \subset B_0$, fix any holomorphic section of $\pi : X_U \to U$ so that we can identify $P_U$ and $X_U$. Hence $\rho$ becomes the addition operation of the abelian scheme $X_U \times_U X_U \to X_U$. With this description, the pullback morphism is fiberwise $\rho^* : H^1(F, \ZZ) \to H^1 (F, \ZZ) \oplus H^1 (F, \ZZ)$, $x \mapsto (x,x)$. Hence the morphism $R^1 \pi_* \ZZ \to R^1 \nu_* \ZZ$ is an isomorphism over $U$, and the claim follows.
	\end{proof}
	
	\begin{remark}
		\begin{enumerate}
			\item A posteriori, one has an interpretation of \Cref{cor:voisin} in terms of \Cref{thm:abelian scheme}. The abelian scheme $\nu$ is projective, and $\nu$ and $\pi$ are fiberwise isomorphic. Hence the smooth fibers of $\pi$ are projective, even when the hyper-K\"ahler manifold $X$ is not.
			
			\item \Cref{thm:abelian scheme} is also related to Oguiso's result \cite{ogu09picard} in the following sense. Consider the generic fiber $P_L \to \Spec L$ of $\nu : P_0 \to B_0$. It is an abelian variety over $L$. Since $P_0$ has a unique polarization (up to scalar), $P_L$ has a unique polarization. Now ampleness is an open condition in $\NS(P_L)_{\RR}$, so the uniqueness of the polarization implies $\rho(P_L) = 1$.
			
			\item If we further assume $X$ is projective, then the discussion becomes algebraic and hence the smooth morphism $\pi : X_0 \to B_0$ admits \'etale local sections. Thus $\pi$ becomes an \'etale torsor under $\nu$.
		\end{enumerate}
	\end{remark}
	
	The unique abelian scheme $P_0$ above should be considered as the identity component of the automorphism scheme of $\pi$ (see \cite[\S 8.3]{ari-fed16} and \Cref{rmk:automorphism scheme} below). Define a sheaf of relative automorphisms acting by translations on each fibers by
	\begin{equation}
		\underline{\Aut}^{tr}_{X_0/B_0} (U) = \{ f : X_U \to X_U : U \mbox{-automorphism acting by translation on each fibers} \} .
	\end{equation}
	
	\begin{proposition} \label{prop:abelian scheme is automorphism scheme}
		The abelian scheme $\nu : P_0 \to B_0$ represents $\underline{\Aut}^{tr}_{X_0/B_0}$.
	\end{proposition}
	\begin{proof}
		This almost follows from the definition. Let us temporarily denote by $\underline{P_0}$ the sheaf of analytic local sections of $\nu$. Since $P_0$ acts on $X_0$ by fiberwise translation, we have a sheaf homomorphism $\underline{P_0} \to \underline{\Aut}^{tr}_{X_0/B_0}$. The homomorphism is injective because the $P_0$-action is effective. Now $X_0$ was in fact a torsor under $P_0$. Over a small analytic open subset $U \subset B_0$, $X_U \to U$ admits a section so it becomes an abelian scheme isomorphic to $P_U \to U$. Hence the set of sections $P_U(U) = X_U(U)$ consists of precisely the translation automorphisms of $X_U$. This proves the homomorphism $\underline{P_0} \to \underline{\Aut}^{tr}_{X_0/B_0}$ is surjective stalkwise. Thus it is an isomorphism.
	\end{proof}
	
	\begin{remark} \label{rmk:automorphism scheme}
		We have later learned that the full relative automorphism sheaf $\SheafAut_{X/B}$ of the Lagrangian fibration is representable by an analytic group scheme $\Aut_{X/B} \to B$. This is essentially a consequence of the existence of the Hilbert scheme of $X \times_B X \to B$. See \cite[Thm 5.23]{nit:fga} for the case when $X$ is projective. The proof for the non-projective case roughly goes as follows. Since we are assuming $B$ is smooth, $\pi$ is flat by the miracle flatness theorem (e.g., \cite[\S 3.20]{fis:cg}). Deduce from \cite{pou69} the existence of the Hilbert scheme $\operatorname{Hilb}_{X \times_B X/B} \to B$, an (infinite) disjoint union of complex spaces proper over $B$. Imitate the proof of \cite[Thm 1.10]{kol:ratcurve} to show a morphism $\SheafAut_{X/B} \to \underline{\operatorname{Hilb}}_{X \times_B X/B}$ sending an automorphism to its graph is an open subfunctor. This proves $\SheafAut_{X/B}$ is representable by an open subspace of a complex space $\operatorname{Hilb}_{X \times_B X/B}$.
		
		Therefore, $\nu : P_0 \to B_0$ is really the identity component of the group scheme $\Aut_{X_0/B_0} \to B_0$ in a precise sense.
	\end{remark}

	\subsection{Polarization type and divisibility of $\pi^* \mathcal O_B(1)$} \label{subsec:polarization type}
	The purpose of this subsection is to study two numerical invariants associated to a Lagrangian fibered hyper-K\"ahler manifold and study their relations: they are the polarization type of $\pi$ in \Cref{def:polarization scheme} and the divisibility of the line bundle $\pi^* \mathcal O_B(1)$. Throughout, we will write $h \in H^2 (X, \ZZ)$ for the first Chern class of $\pi^* \mathcal O_B(1)$ and $\operatorname{div}(h)$ for the divisibility \eqref{eq:divisibility} of $h$.
	
	The polarization type of $\pi$ is an $n$-tuple of positive integers $(d_1, \cdots, d_n)$ with $d_1 \mid \cdots \mid d_n$ such that each fiber of the polarization scheme $K_0$ is isomorphic to $(\ZZ/d_1 \oplus \cdots \oplus \ZZ/d_n)^{\oplus 2}$. Since we are assuming the polarization $\lambda : P_0 \to \check P_0$ is primitive, we always have $d_1 = 1$. The polarization type was already computed for all currently known deformation types of hyper-K\"ahler manifolds. The computations were based on its original definition in \cite{wie16}. Therefore, to use the previous results we first need to show our definition is equivalent to the original one.
	
	\begin{lemma} \label{lem:definition of polarization type is equivalent}
		The polarization type in \Cref{def:polarization scheme} is equivalent to the definition in \cite{wie16}.
	\end{lemma}
	\begin{proof}
		Recall our definition of the polarization is constructed as the primitive morphism $(R^2\pi_*\ZZ)^{\vee} \to \ZZ$ of local systems, or equivalently a primitive morphism $\ZZ \to R^2\pi_* \ZZ$. Fix any smooth fiber $F$ of $\pi$. As a local system, $R^2\pi_*\ZZ$ is identified with a $\ZZ$-module $H^2(F,\ZZ)$ with a monodromy $\pi_1(B_0)$-action. In this setting, the primitive morphism $\ZZ \to R^2\pi_* \ZZ$ of local systems corresponds to a primitive homomorphism $\ZZ \to H^2 (F, \ZZ)$ of $\pi_1(B_0)$-modules. In other words, our definition of the polarization type is equivalent to the primitive polarization type of a single smooth fiber $F$ coming from the image of $H^2 (X, \QQ) \to H^2 (F, \QQ)$ by the global invariant cycle theorem. This was precisely how Wieneck defined the polarization type.
	\end{proof}
	
	We can now use the previous results on computations of the polarization type of $\pi$. The following theorem collects all possible polarization types that can occur for known deformation types of hyper-K\"ahler manifolds. For $\text{K3}^{[n]}$ and $\Kum_n$-types the computations are done by \cite{wie16, wie18}. For OG10 and OG6, the computations are contained in \cite{mon-ono22} and \cite{mon-rap21}, respectively.
	
	\begin{theorem} [{\cite{wie16, wie18}, \cite{mon-ono22}, \cite{mon-rap21}}] \label{thm:polarization type computation}
		Let $\pi : X \to B$ be a Lagrangian fibered compact hyper-K\"ahler manifold. Then the polarization type of $\pi$ is
		\[\begin{cases}
			(1,\cdots,1) \quad & \mbox{if } X \mbox{ is of } \text{K3}^{[n]} \mbox{-type;} \\
			(1,1,1,1,1) \quad & \mbox{if } X \mbox{ is of OG10-type;} \\
			(1,\cdots,1, d_1, d_2) \quad & \mbox{if } X \mbox{ is of } Kum_n \mbox{-type; and} \\
			(1,2,2) \quad & \mbox{if } X \mbox{ is of OG6-type}.
		\end{cases}\]
		When $X$ is of $\Kum_n$-type, we set $d_1 = \operatorname{div}(h)$ in $H^2 (X, \ZZ)$ and $d_2 = \frac{n+1}{d_1}$.
	\end{theorem}
	
	It is also important for us that the polarization type is deformation invariant on $\pi$ (see \cite[Thm 1.1]{wie16}). We will later recover this result in \Cref{cor:polarization type is deformation invariant}. Observe in \Cref{thm:polarization type computation} that we have an equality $c_X = d_1 \cdots d_n$ for all known deformation types of hyper-K\"ahler manifolds. In this sense, we expect the polarization type should be considered as a refinement of the Fujiki constant $c_X$. This is also related to the non-primitveness of the image of the restriction homomorphism $H^2 (X, \ZZ) \to H^2 (F, \ZZ)$.
	
	\begin{proposition} \label{prop:non primitive polarization}
		Assume we have an equality $c_X = d_1 \cdots d_n$. Then the image of the restriction homomorphism $H^2 (X, \ZZ) \to H^2 (F, \ZZ)$ in \Cref{lem:matsushita} is generated by $a \theta$, where $a = \operatorname{div}(h)$ and $\theta$ is a primitive ample class representing the canonical polarization of $F$.
	\end{proposition}
	\begin{proof}
		Choose a cohomology class $x \in H^2 (X, \ZZ)$ with $q(h, x) = a$. By \Cref{lem:matsushita}, the class $x_{|F} \in H^2 (F, \ZZ)$ must be a positive integer multiple of the primitive polarization class $\theta$. Set $x_{|F} = b \theta$. Now the claim directly follows from the Fujiki relation
		\[ d_1 \cdots d_n = \frac{1}{n!} \int_F \theta^n = \frac{1}{n!} \int_X h^n \big( \tfrac{1}{b} x \big)^n = c_X \cdot q \big(h, \tfrac{1}{b} x \big)^n = c_X \left( \frac{a}{b} \right)^n . \qedhere\]
	\end{proof}
	
	Though not used in this paper, the divisibility of $h$ is also related to the existence of a rational section of $\pi$. We end this subsection with the following observation.
	
	\begin{proposition}
		Assume $c_X = d_1 \cdots d_n$ and $\pi$ admits at least one rational section. Then $\operatorname{div}(h) = 1$ or $2$.
	\end{proposition}
	\begin{proof}
		If $\pi$ admits a rational section, then $X_0 \cong P_0$ becomes a projective abelian scheme (\Cref{ex:when pi has a rational section}). By the general theory of abelian schemes, twice a polarization is always associated to a line bundle (e.g., \cite[Prop 6.10]{mum:git} or \cite[Def I.1.6]{fal-chai:abelian}). This means $2\theta \in H^2 (F, \ZZ)$ is contained in the image of $\Pic(X) \subset H^2 (X, \ZZ) \to H^2 (F, \ZZ)$. By \Cref{prop:non primitive polarization}, this implies $\operatorname{div}(h) = 1$ or $2$. 
	\end{proof}
	
	If $X$ is of $\text{K3}^{[n]}$ or $\Kum_n$-type then its Lagrangian fibration $\pi : X \to B$ may have $\operatorname{div}(h) > 2$. In such cases, $\pi$ (and any of its deformation) would never admit any rational section and the notion of the $P_0$-torsor is necessary.

	\subsection{The polarization scheme and $H^2$-trivial automorphisms}
	We present the proof of \Cref{prop:polarization scheme ver1} and \ref{prop:polarization scheme ver2} in this section.
	
	\begin{lemma} \label{lem:rational section}
		Any rational section of $\nu : P_0 \to B_0$ can be uniquely extended to an honest section.
	\end{lemma}
	\begin{proof}
		Assume $s : B_0 \dashrightarrow P_0$ is a rational section undefined at $b \in B_0$. Let $S \subset P_0$ be the closure of the image of $s$, so that we obtain a proper birational morphism $\nu_{|S} : S \to B_0$. Since $B_0$ is smooth and $s$ is undefined at $b$, the fiber $S_b = (\nu_{|S})^{-1}(b)$ is a uniruled variety (e.g., \cite[Thm VI.1.2]{kol:ratcurve}). This means an abelian variety $\nu^{-1}(b)$ contains a uniruled variety $S_b$. Contradiction. See \cite[Cor 8.4.6]{neron} for an alternative proof.
	\end{proof}
	
	\begin{remark} \label{rmk:rational section is defined over B0}
		The same argument applies to $\pi$ and proves the following: any rational section of $\pi$ is necessarily defined over $B_0$.
	\end{remark}
	
	\begin{proposition} \label{prop:aut defines global section of P}
		Every $H^2$-trivial automorphism in $\Aut^{\circ} (X/B)$ defines a global section of $P_0 \to B_0$. That is, we have a closed immersion of group schemes
		\[ \Aut^{\circ} (X/B) \hookrightarrow P_0 .\]
	\end{proposition}
	\begin{proof}
		Recall from \Cref{prop:abelian scheme is automorphism scheme} that $P_0$ is the abelian scheme representing the translation automorphism sheaf $\underline{\Aut}^{tr}_{X_0/B_0}$. Hence our goal is to prove $\Aut^{\circ}(X/B)$ acts on $\pi : X_0 \to B_0$ by fiberwise translation automorphisms. Consider the quotient $\bar X = X / \Aut^{\circ}(X/B)$ with a commutative diagram
		\[\begin{tikzcd}[row sep=tiny]
			X \arrow[rd, "p"] \arrow[dd, "\pi"] \\
			& \bar X \arrow[ld, "\bar \pi"] \\
			B
		\end{tikzcd}.\]
		We first claim $p$ is \'etale on general fibers over $B$. Let $S \subset X$ be the ramified locus of $p$. It has codimension $\ge 2$ because $p$ is quasi-\'etale by \Cref{prop:appendix 1}. Let $b \in B$ be a general point, so that the fibers $F = X_b$ and $\bar F = \bar X_b$ are both smooth. Observe the ramification locus of $p : F \to \bar F$ is precisely $S \cap F$, which is of codimension $\ge 2$ since $b$ is general. The purity of the branch locus theorem forces $p : F \to \bar F$ to be \'etale.
		
		Now we have a finite \'etale quotient $p : F \to \bar F = F / \Aut^{\circ} (X/B)$ between smooth projective varieties. Its Galois group $\Aut^{\circ} (X/B)$ acts on $F$ by fixed point free automorphisms. Since $F$ and $\bar F$ are both abelian varieties (\cite[Thm 3]{sch20}), this means $\Aut^{\circ} (X/B)$ acts on $F$ by translations. The conclusion is that on a general fiber of $\pi$, the group $\Aut^{\circ}(X/B)$ acts by translation. Finally, by \Cref{prop:abelian scheme is automorphism scheme} this means $\Aut^{\circ}(X/B)$ defines a rational section of $\nu : P_0 \to B_0$. By \Cref{lem:rational section}, the rational section must be defined over the entire $B_0$ and becomes an honest section. Hence $\Aut^{\circ}(X/B)$ acts by translations over the entire $B_0$.
	\end{proof}
	
	An immediate byproduct is that $\Aut^{\circ}(X/B)$ is abelian.
	
	\begin{proposition} \label{prop:aut0 is abelian}
		$\Aut^{\circ}(X/B)$ is a finite abelian group. \qed
	\end{proposition}
	
	We next understand the behavior of the polarization $\lambda$ under deformations of $\pi$. A related result is Wieneck's deformation invariance of the polarization type of $\pi$ \cite[Thm 1.1]{wie16}. Recall that the polarization scheme $K_0$ was defined to be the kernel of the polarization $\ker \lambda$ over $B_0$. To deal with a more technical \Cref{prop:polarization scheme ver1} and \ref{prop:polarization scheme ver2}, we have defined $K_0[a] = \ker (a\lambda)$ for each positive integer $a$:
	\[\begin{tikzcd}
		0 \arrow[r] & K_0[a] \arrow[r] & P_0 \arrow[r, "a\lambda"] & \check P_0 \arrow[r] & 0
	\end{tikzcd}.\]
	Here the morphism $a \lambda$ is a composition $P_0 \to P_0 \to \check P_0$ of the multiplication by $a$ endomorphism and $\lambda$. Since the abelian scheme $P_0$ was associated to the VHS $(R^1\pi_*\ZZ)^{\vee}$, there is a VHS version of this sequence
	\begin{equation} \label{eq:polarization_vhs_ses}
		\begin{tikzcd}
			0 \arrow[r] & (R^1 \pi_* \ZZ)^{\vee} \arrow[r, "a\lambda_*"] & R^1 \pi_* \ZZ \arrow[r] & \underline {K_0[a]} \arrow[r] & 0
		\end{tikzcd} .
	\end{equation}
	The cokernel $\underline {K_0[a]}$ is a local system of finite abelian groups on $B_0$ and is related to $K_a$ as follows: $\underline {K_0[a]}$ is a sheaf of analytic sections of the group scheme $K_0[a] \to B_0$, and $K_0[a]$ is the total space of the local system $\underline {K_0[a]}$. Therefore, we can relate $K_0[a]$ to either abelian schemes or variation of Hodge structures. This technical flexibility will be useful to describe deformation behaviors of $K_0[a]$.
	
	\begin{lemma} \label{lem:deformation of polarization scheme}
		Let $p : \mathcal X \xrightarrow{\pi} \mathcal B \xrightarrow{q} \Delta$ be a family of Lagrangian fibered compact hyper-K\"ahler manifolds over a complex open disc $\Delta$. Set $\mathcal B_0 \subset \mathcal B$ the smooth locus of $\pi$. Then for each positive integer $a$, there exists a finite \'etale group scheme $\mathcal K_0[a]$ over $\mathcal B_0$ parametrizing the group schemes $K_a$ over $(B_0)_t$ for all $t \in \Delta$.
	\end{lemma}
	\begin{proof}
		Let $\mathcal X_0 = \pi^{-1} (\mathcal B_0)$ be the preimage of $\mathcal B_0$, so that the restriction $\pi_0 : \mathcal X_0 \to \mathcal B_0$ of $\pi$ is a smooth proper family of abelian varieties. Consider the local system $R^2 \pi_{0,*} \ZZ$ on $\mathcal B_0$. Our first claim is $H^0 (\mathcal B_0, R^2 \pi_{0,*} \ZZ) \cong \ZZ$. Denoting by $j : \mathcal B_0 \to \mathcal B$ an open immersion, it is enough to prove $H^0 (\Delta, q_* j_* R^2 \pi_{0, *} \ZZ) \cong \ZZ$. Notice that $q_* j_* R^2 \pi_{0, *} \ZZ$ is a constructible sheaf, because $R^2 \pi_{0, *} \ZZ$ is a local system, its pushforward by $j_*$ is a constructible sheaf on $\mathcal B$ (e.g., \cite[Ex VIII.10]{kas-sch:sheaves}), and again its pushforward by $q_*$ is a constructible sheaf on $\Delta$. For each $t \in S$, we have a Lagrangian fibered hyper-K\"ahler manifold $\pi : X_t \to B_t$ and we may apply our previous discussions
		\[ H^0 ((B_0)_t, R^2 \pi_* \ZZ) \cong \ZZ .\]
		This proves every fiber of $q_* j_* R^2 \pi_{0, *} \ZZ$ is isomorphic to $\ZZ$. In this setting, we will formally prove the sheaf has $\ZZ$ global sections in \Cref{lem:constructible sheaf}. This proves the claim $H^0 (\mathcal B_0, R^2 \pi_{0,*} \ZZ) \cong \ZZ$.
		
		We have a unique primitive morphism $(R^2 \pi_{0, *} \ZZ)^{\vee} \to \ZZ$ of local systems on $\mathcal B_0$. As $\pi_0$ is a family of abelian varieties, we have an isomorphism $R^2 \pi_{0, *} \ZZ = \wedge^2 R^1 \pi_{0, *} \ZZ$. This gives us a unique primitive morphism of local systems (in fact, a polarization of VHS by \Cref{cor:voisin}) $(R^1\pi_{0, *}\ZZ)^{\vee} \otimes (R^1\pi_{0, *}\ZZ)^{\vee} \to \ZZ$ over $\mathcal B_0$. Consider the scalar multiple $a$ of it and induce a morphism $(R^1\pi_{0, *}\ZZ)^{\vee} \to R^1\pi_{0, *}\ZZ$ whose cokernel $\underline {\mathcal K_0[a]}$ is a local system on $\mathcal B_0$, parametrizing the family of local systems $\underline {K_0[a]}$ on each $(B_0)_t$ in \eqref{eq:polarization_vhs_ses}. The total space of $\underline {\mathcal K_0[a]}$ gives our desired finite \'etale group scheme $\mathcal K_0[a] \to \mathcal B_0$.
	\end{proof}
	
	\begin{lemma} \label{lem:constructible sheaf}
		Let $F$ be a constructible sheaf on a complex open disc $\Delta$. If every fiber of $F$ is isomorphic to $\ZZ$, then we have $H^0 (\Delta, F) \cong \ZZ$.
	\end{lemma}
	\begin{proof}
		Since $F$ is constructible, $F_{|U}$ is a local system on a completement $U$ of a finite set of points $t_1, \cdots, t_k \in \Delta$. Let $U_i \subset U$ be a small punctured disc around $t_i$. The restriction $F_{|U_i}$ is determined by the representation
		\[ \rho_i : \ZZ \cong \pi_1(U_i) \to \Aut(\ZZ) = \{ \pm 1 \} .\]
		We have only two possibilities $\rho_i (1) = \pm 1$ for each $i$. Suppose we have $\rho_i(1) = -1$ for some $i$. Consider the total space $f : \operatorname{Et}(F) \to \Delta$ of the entire constructible sheaf $F$ (espace \'etal\'e). The map $f$ is holomorphic and \'etale, i.e., a local isomorphism. The condition $\rho_i(1) = -1$ geometrically translates to the fact that $f^{-1} (U_i)$ consists of a single copy of $U_i$ (the zero section) and infinite number of two-sheeted coverings of the punctured disc $U_i$. By the very assumption, the preimage $f^{-1} (t_i) = \{ p_1, p_2, \cdots \}$ should be isomorphic to $\ZZ$. Since $f$ is a local isomorphism, there should be an open disc neighborhood of each $p_i \in \operatorname{Et}(F)$. Along the two-sheeted coverings of $U_i$ in $f^{-1} (U_i)$, this cannot happen. Therefore, the only possibility is that all $p_i$ are the non-Hausdorff points filling in the unique punctured disc component in $f^{-1} (U_i)$ (i.e., the zero section). Hence we obtain at least $\ZZ$ global sections around the zero section and we are done.
		
		The remaining case is when $\rho_i(1) = 1$ for all $i$. This means $F_{|U}$ is a constant sheaf $\ZZ$. The reader should be aware that this does not imply $F$ is a constant sheaf $\ZZ$ on $\Delta$. This can be again conveniently seen in the total space $f : \operatorname{Et}(F) \to \Delta$. Although $f$ is a local homeomorphism, it is not a covering space unless $\operatorname{Et}(F)$ is Hausdorff. Indeed, the fibers $f^{-1}(t_i)$ can consist of non-Hausdorff points in $\operatorname{Et}(F)$ and this gives us a classification of such a constructible sheaf $F$. In any case, there are always $\ZZ$ global sections.
	\end{proof}
	
	\Cref{lem:deformation of polarization scheme} in particular recovers \cite[Thm 1.1]{wie16}.
	
	\begin{corollary} \label{cor:polarization type is deformation invariant}
		The polarization type of $\pi$ is invariant under deformations of $\pi$. \qed
	\end{corollary}
	
	The following final observation is elementary but nontrivial. We match its notation to our original discussion.
	
	\begin{lemma} \label{lem:torsion section property}
		Let $\mathcal P_0 \to \mathcal B_0$ be an abelian scheme over a complex manifold $\mathcal B_0$ and $a\lambda : \mathcal P_0 \to \check {\mathcal P}_0$ a polarization with $\mathcal K_0[a] = \ker (a\lambda)$. Assume there exists a torsion section $f : \mathcal B_0 \to \mathcal P_0$. If $f(\mathcal B_0) \cap \mathcal K_0[a] \neq \emptyset$ then $f(\mathcal B_0) \subset \mathcal K_0[a]$.
	\end{lemma}
	\begin{proof}
		The statement is topological and local on the base $\mathcal B_0$, so we may assume $\mathcal B_0$ is a complex open ball $S$ and $\mathcal P_0 \to \mathcal B_0$ is homeomorphic to the topological constant group scheme $(\RR/\ZZ)^{2n} \times S \to S$. In this setting, the kernel $\mathcal K_0[a]$ is a constant subgroup scheme and the \emph{torsion} section $f$ is a constant section. Hence $f(S) \cap \mathcal K_0[a] \neq \emptyset$ if and only if $f(S) \subset \mathcal K_0[a]$.
	\end{proof}
	
	\begin{proof} [Proof of \Cref{prop:polarization scheme ver1}]
		Consider a one-parameter family of Lagrangian fibered hyper-K\"ahler manifolds $\mathcal X \to \mathcal B \to \Delta$ over a complex disc $\Delta$. By \Cref{lem:deformation of polarization scheme}, there exists a notion of a family of abelian schemes $\mathcal P_0 \to \mathcal B_0$ and a family of finite \'etale group schemes $\mathcal K_0[a] \subset \mathcal P_0$. \Cref{prop:aut defines global section of P} proves we have a closed immersion $\Aut^{\circ}(X/B) \hookrightarrow P_0$ for a single fiber. In fact, the argument applies to the entire family and produces $\Aut^{\circ}(X/B)$ global sections of $\mathcal P_0 \to \mathcal B_0$, or equivalently an embedding
		\[ \Aut^{\circ}(X/B) \hookrightarrow \mathcal P_0 .\]
		Since $\Aut^{\circ}(X/B)$ is finite, the global sections are torsion. Suppose we had $\Aut^{\circ}(X/B) \hookrightarrow K_0[a]$ for the original Lagrangian fibration over $0 \in \Delta$. Then this forces $\Aut^{\circ}(X/B) \hookrightarrow \mathcal K_0[a]$ over the entire $\Delta$ by \Cref{lem:torsion section property}. The claim follows.
	\end{proof}
	
	\begin{proof} [Proof of \Cref{prop:polarization scheme ver2}]
		Recall from \Cref{prop:non primitive polarization} that the restriction map $H^2 (X, \ZZ) \to H^2 (F, \ZZ)$ has a rank $1$ image generated by the class $a\theta$, where $a = \operatorname{div}(h)$ and $\theta$ is the primitive ample class corresponding to our polarization $\lambda : F \to \check F$. The preimage of $a\theta \in H^2 (F, \ZZ)$ under this restriction homomorphism is precisely $S = \{ x \in H^2 (X, \ZZ) : q(x, h) = a \}$. By \Cref{prop:polarization scheme ver1}, the claim is invariant under deformations of $\pi$. We may thus deform $\pi$ and assume $\Pic(X) \cap S \neq \emptyset$. In other words, we may assume the composition $\Pic(X) \subset H^2 (X, \ZZ) \to H^2 (F, \ZZ)$ is generated by $a\theta$.
		
		The assertion $\Aut^{\circ}(X/B) \hookrightarrow K_0[a] = \ker (a\lambda)$ is equivalent to $a\lambda (\Aut^{\circ}(X/B)) = 0$. The latter equality may be verified fiberwise, so we may concentrate on a single fiber $F = \nu^{-1}(b) = \pi^{-1}(b)$. Let $L$ be any line bundle on $X$ such that its image under $\Pic(X) \to H^2 (F, \ZZ)$ is $a\theta$. This means the polariztion $a\lambda$ can be described as
		\[ a\lambda : F \to \check F, \qquad t_x \mapsto [t_x^* (L_{|F}) \otimes L_{|F}^{-1} ] .\]
		If we assume $t_x = f_{|F}$ is from a global $H^2$-trivial automorphism $f \in \Aut^{\circ}(X/B)$, then we have a sequence of identities
		\[ t_x^* (L_{|F}) = (f_{|F})^* (L_{|F}) = (f^* L)_{|F} \cong L_{|F} ,\]
		where the last isomorphism follows from the fact $f$ acts on $\Pic(X) \subset H^2 (X, \ZZ)$ trivially. This proves $a\lambda$ sends $\Aut^{\circ}(X/B)$ to $0$ and the claim follows.
	\end{proof}

	\section{The minimal split covering and $H^2$-trivial automorphisms} \label{sec:minimal split covering}
	This section discusses an explicit construction of certain $H^2$-trivial automorphisms. This will be conveniently used in the next section when we describe the $\Aut^{\circ} (X)$-action explicitly for certain examples of $\Kum_n$-type hyper-K\"ahler manifolds. Recall that the group $\Aut^{\circ} (X)$ is computed for all known deformation types of hyper-K\"ahler manifolds; Beauville \cite{bea83b} for $\text{K3}^{[n]}$-types, Boissi\`ere--Nieper-Wi{\ss}kirchen--Sarti \cite{boi-nei-sar11} for $\Kum_n$-types, and Mongardi--Wandel \cite{mon-wan17} for OG10 and OG6-types. The strategy is to compute the group for a specific choice of a complex structure and then use Hassett--Tschinkel's deformation equivalence \cite[Thm 2.1]{has-tsc13}. Unfortunately, this argument doesn't tell us how $\Aut^{\circ}(X)$ exactly acts on $X$ for the deformations. The goal of this section is to introduce \Cref{prop:galois} to partially resolve this problem.
	
	Throughout the section, we stick to the following setting. Let $M$ be a \emph{projective} holomorphic symplectic manifold, not necessarily irreducible. By Beauville--Bogomolov decomposition theorem, $M$ must admit a finite \'etale covering $X \times T \to M$, called a \emph{split covering}, where $X$ is a finite product of projective hyper-K\"ahler manifolds and $T$ is an abelian variety. In fact, Beauville in \cite[\S 3]{bea83b} also considered the smallest possible minimal covering. A \emph{minimal split covering} of $M$ is the smallest possible split covering of $M$, in the sense that every split covering factors through it. The minimal split covering of $M$ always exists and is unique up to a (non-unique) isomorphism. Moreover, it is a Galois covering. We refer to Beauville's original paper for more details about minimal split coverings.
	
	Meanwhile, Kawamata \cite[Thm 8.3]{kaw85} proved that if $M$ is a K-trivial smooth projective variety then its Albanese morphism $\Alb : M \to \Alb(M)$ has to be an \'etale fiber bundle. More concretely, there exists an isogeny $\phi : T \to \Alb(M)$ of abelian varieties such that the base change of $\Alb$ becomes a trivial fiber bundle over $T$. We obtain a cartesian diagram
	\begin{equation} \label{diag:kawamata}
	\begin{tikzcd}
		X \times T \arrow[r, "\Phi"] \arrow[d, "\pr_2"] & M \arrow[d, "\Alb"] \\
		T \arrow[r, "\phi"] & \Alb(M)
	\end{tikzcd} ,
	\end{equation}
	where $X$ is a fiber of the Albanese morphism. In particular, one sees $\Phi: X \times T \to M$ becomes a split covering of $M$.
	
	Combining the two discussions, we obtain:
	
	\begin{proposition} \label{prop:minimal isogeny}
		Let $M$ be a projective holomorphic symplectic manifold and $\Alb : M \to \Alb(M)$ its Albanese morphism, an \'etale fiber bundle by Kawamata. Assume $X = \Alb^{-1}(0)$ is a projective hyper-K\"ahler manifold. Then there exists a unique isogeny $\phi : T \to \Alb(M)$ such that the morphism $\Phi$ in the fiber diagram \eqref{diag:kawamata} becomes the minimal split covering of Beauville.
	\end{proposition}
	\begin{proof}
		Use Kawamata's result to construct an isogeny $\phi' : T' \to \Alb(M)$ trivializing the Albanese map as in \eqref{diag:kawamata}. Since $\phi'$ is a finite Galois covering, $\Phi'$ is also a finite Galois covering with $\Gal(\Phi') \cong \Gal(\phi')$. The first lemma in \cite[\S 3]{bea83b} claims $\Aut (X \times T') = \Aut(X) \times \Aut(T')$. Hence the $\Gal(\Phi')$-action on $X \times T'$ is by $(f, a)$ where $f$ and $a$ are automorphisms on $X$ and $T$, respectively. The isomorphism $\Gal(\Phi') \to \Gal(\phi')$ is by the second projection $(f, a) \mapsto a$. Since $\Gal(\phi')$ is the kernel of the isogeny $\phi'$, the automorphisms $a$ must be translations of $T'$.
		
		Now consider the homomorphism $\Gal(\Phi') \to \Aut(X)$ by $(f, a) \mapsto f$. Set $H$ by the kernel of it; it consists of elements of the form $(\id_X, a)$. Under the isomorphism $\Gal (\Phi') \cong \Gal(\phi')$, we can consider it as a subgroup of $\Gal(\phi')$, so there exists a Galois covering $T' \to T = T'/H$ corresponding to it. Let $\phi : T \to \Alb(M)$ be the morphism factorizing $\phi'$. We have a cartesian diagram
		\[\begin{tikzcd}
			X \times T' \arrow[r] \arrow[d, "\pr_2"] & X \times T \arrow[r, "\Phi"] \arrow[d, "\pr_2"] & M \arrow[d, "\Alb"] \\
			T' \arrow[r] & T \arrow[r, "\phi"] & \Alb(M)
		\end{tikzcd}.\]
		By construction, $\Gal(\Phi)$ consists of automorphisms $(f, a)$ with no $(\id_X, a)$ (i.e., the $\Gal(\phi)$-action on $X$ is effective). But this means $\Phi$ is precisely Beauville's minimal split covering \cite[\S 3]{bea83b}. The uniqueness of $\phi$ follows from the uniqueness of the minimal split covering.
	\end{proof}
	
	The proposition in particular proves that the minimal split covering can be always realized by an isogeny $\phi : T \to \Alb(M)$ and the base change \eqref{diag:kawamata}.
	
	\begin{definition} \label{def:minimal isogeny}
		We call $\phi : T \to \Alb(M)$ in \Cref{prop:minimal isogeny} the \emph{minimal isogeny} trivializing the Albanese morphism $\Alb : M \to \Alb(M)$. It is unique up to a (non-unique) isomorphism.
	\end{definition}
	
	In fact, the proof of \Cref{prop:minimal isogeny} is saying more about an arbitrary isogeny $\phi'$.
	
	\begin{corollary} \label{cor:minimal isogeny}
		Notation as in \Cref{prop:minimal isogeny}. Let $\phi' : T' \to \Alb(M)$ be any isogeny trivializing the Albanese morphism. Then
		\begin{enumerate}
			\item $\phi'$ factors though the minimal isogeny $\phi$.
			\item There exists a canonical $\Gal(\phi')$-action on $X$.
			\item The isogeny $\phi'$ is minimal if and only if the $\Gal(\phi')$-action on $X$ is effective.
		\end{enumerate}
	\end{corollary}
	\begin{proof}
		All of these can be directly deduced from the proof of \Cref{prop:minimal isogeny}. Recall $\Gal(\Phi') \to \Gal(\phi')$, $(f,a) \mapsto a$ is an isomorphism. Therefore, $f = f_a$ is uniquely determined by $a$, and this defines $\Gal(\phi') \to \Aut(X)$, $a \mapsto f_a$.
	\end{proof}
	
	Now we can state the main result of this section. The ideas here were already contained in \cite{bea83b, bea83}.
	
	\begin{proposition} \label{prop:galois}
		Notation as in \Cref{prop:minimal isogeny} and \ref{cor:minimal isogeny}. Then $\Gal(\phi')$ acts on $X$ by $H^2$-trivial automorphisms. That is, we have a canonical homomorphism
		\[ \Gal(\phi') \to \Aut^{\circ} (X) ,\]
		which is injective if and only if $\phi'$ is minimal.
	\end{proposition}
	\begin{proof}
		By \Cref{cor:minimal isogeny}, we may assume $\phi' = \phi$ is minimal and $\Gal(\phi) \subset \Aut(X)$. The content of the proposition is that it is further a subgroup of $\Aut^{\circ}(X)$.
		
		Consider the diagram \eqref{diag:kawamata}. Our first step is to equip $T$-actions on all the four spaces to make the diagram $T$-equivariant. Equip a $T$-action on $T$ by translation, and on $X \times T$ only on the second factor again by translation. The $T$-action on $\Alb(M)$ is by translation via the morphism $\phi$: if $a \in T$ and $z \in \Alb(M)$ then we define $a.z = z + \phi(a)$.
		
		To equip a $T$-action on $M$, we claim the $T$-action on $X \times T$ descends to $M$ via $\Phi$. The descent works if the $\Gal(\Phi)$-action on $X \times T$ commutes with the $T$-action. Recall from the discussions in \Cref{prop:minimal isogeny} that $\Gal(\Phi)$ acts on $X \times T$ by $(f, a)$ where $f$ is an automorphism of $X$ and $a$ is a translation of $T$. Let $b \in T$ and $(x, t) \in X \times T$. Then we have a sequence of identities
		\[ b . \Big( (f, a) . (x, t) \Big) = (f(x), t + a + b) = (f, a) . \Big( b . (x,t) \Big) .\]
		This proves the $T$-action and $\Gal(\Phi)$-action commutes, yielding the descent $T$-action on $M$. The conclusion is that $\Alb$ becomes automatically $T$-equivariant (and hence the diagram \eqref{diag:kawamata} becomes $T$-equivariant).
		
		By definition, the stabilizer the $T$-action on $\Alb(M)$ is precisely $\ker \phi = \Gal(\phi)$. Since the Albanese map $\Alb : M \to \Alb(M)$ is $T$-equivariant, this induces a $\Gal(\phi)$-action on the fiber $\Alb^{-1} (0) = X$. One easily shows this coincides with our previous $\Gal(\phi)$-action on $X$. Notice that any $T$-action on $M$ is isotopic to the identity map because $T$ is path connected. In particular, $T$ acts on $M$ trivially at the level of cohomology $H^* (M, \QQ)$. The embedding $X \subset M$ is $\Gal(\phi)$-equivariant, so we have a $\Gal(\phi)$-equivariant restriction homomorphism
		\[ H^2 (M, \QQ) \to H^2 (X, \QQ) .\]
		Hence it suffices to prove this restriction homomorphism is surjective.
		
		The question now became topological. Deform the complex structure of the hyper-K\"ahler manifold $X$ very generally so that $H^2 (X, \QQ)$ becomes a simple $\QQ$-Hodge structure (we will have to lose the projectiveness of $X$). The complex structure of $M$ can be correspondingly chosen in a way that the finite covering map $\Phi : X \times T \to M$ becomes holomorphic. Therefore, the Hodge structure morphism $H^2 (M, \QQ) \to H^2 (X, \QQ)$ is either $0$ or surjective. We only need to rule out the former possibility.
		
		To prove it is nonzero, consider any global holomorphic symplectic form $\sigma$ on $M$. Pulling it back to $X \times T$ gives a global holomorphic symplectic form on $X \times T$. But $H^{2,0} (X \times T) = H^{2,0} (X) \oplus H^{2,0} (T)$ by K\"unneth. If $\sigma$ was $0$ in the $H^{2,0}(X)$-component then this would mean $\sigma$ doesn't contain any $2$-forms along the tangent direction of $X$, violating $\sigma$ is a symplectic form. Hence $\sigma_{|X}$ cannot be $0$. The claim follows.
	\end{proof}
	
	\begin{remark} \label{rmk:aut0 new definition}
		An alternative way to state the results in this section is as follows. Any isogeny $\phi' : T' \to \Alb(M)$ trivizalizing the Albanese morphism defines a group homomorphism $\Gal (\phi') \to \Aut^{\circ}(X)$. The image of this homomorphism is independent on the choice of $\phi'$, which we denote by
		\[ \Aut' (X) \subset \Aut^{\circ}(X) .\]
		It is a finite abelian group, isomorphic to $\Gal(\phi)$ for a minimal isogeny $\phi$, and is deformation invariant on $X$. For example, we will later see that when $X$ is of $\Kum_n$-type then
		\[ \Aut'(X) \cong (\ZZ/n+1)^{\oplus 4}, \qquad \Aut^{\circ} (X) \cong \ZZ/2 \ltimes (\ZZ/n+1)^{\oplus 4} .\]
		Our main result can be more directly stated with this definition. See \Cref{rmk:polarization scheme equality}.
	\end{remark}

	\section{The $H^2$-trivial automorphisms and polarization scheme for generalized Kummer varieties} \label{sec:aut0 computation}
	The goal of this section is an explicit computation of the group $\Aut^{\circ}(X/B)$ and the polarization scheme $K_0$ for certain Lagrangian fibrations of $\Kum_n$-type hyper-K\"ahler manifolds. Since the group $\Aut^{\circ}(X/B)$ will be of interest for all known deformation types of hyper-K\"ahler manifolds, we state the result in a more general form. The following is the first main theorem of this section.
	
	\begin{theorem} \label{thm:aut0 computation}
		Let $\pi : X \to B$ be a Lagrangian fibration of a compact hyper-K\"ahler manifold.
		\begin{enumerate}
			\item $\Aut^{\circ} (X/B) \cong \begin{cases}
				\{ \id \} & \mbox{if } \ X \mbox{ is of K3}^{[n]} \mbox{ or OG10-type,} \\
				(\ZZ/2)^{\oplus 4} \qquad & \mbox{if } \ X \mbox{ is of OG6-type.}
			\end{cases}$
			
			\item Assume $X$ is of $\Kum_n$-type and $(1, \cdots, 1, d_1, d_2)$ is the polarization type of $\pi$ in \Cref{thm:polarization type computation}. Then
			\[\Aut^{\circ} (X/B) \cong \begin{cases}
				(\ZZ/2)^{\oplus 5} & \mbox{if } n = 3 \mbox{ and the polarization type is } (1, 2, 2) ,\\
				(\ZZ/d_1 \oplus \ZZ/d_2)^{\oplus 2} \quad & \mbox{otherwise.}
			\end{cases}\]
		\end{enumerate}
	\end{theorem}
	
	Notice that the bigger group $\Aut^{\circ} (X)$ is already trivial for $\text{K3}^{[n]}$ and OG10-types (see \cite{bea83b} and \cite{mon-wan17}), so the theorem is clear in these cases. For $\Kum_n$ and OG6-types, recall from \Cref{thm:aut0 is deformation invariant} that $\Aut^{\circ}(X/B)$ is deformation invariant on $\pi$. By \cite[\S 6.28]{wie18}, every Lagrangian fibration of a $\Kum_n$-type hyper-K\"ahler manifold is deformation equivalent to the \emph{moduli of sheaves construction}, which will be recalled in \Cref{subsec:moduli construction Kum}. By \cite{mon-rap21}, every Lagrangian fibration of an OG6-type hyper-K\"ahler manifold is deformation equivalent to each other. Therefore, \Cref{thm:aut0 computation} follows from the following more concrete results.
	
	\begin{proposition} \label{prop:aut0 computation for Kum}
		Let $\pi : X \to B$ be a Lagrangian fibration of a $\Kum_n$-type hyper-K\"ahler manifold, obtained by the moduli of sheaves construction from a triple $(S, l, s)$ in \Cref{def:moduli construction Kum}. Let $(d_1, d_2)$ be the polarization type of the ample class $l$. Then
		\[\Aut^{\circ} (X/B) \cong \begin{cases}
			(\ZZ/2)^{\oplus 5} & \mbox{if } n = 3 \mbox{ and } d_1 = d_2 = 2 ,\\
			(\ZZ/d_1 \oplus \ZZ/d_2)^{\oplus 2} \quad & \mbox{otherwise.}
		\end{cases}\]
	\end{proposition}
	
	\begin{proposition} [Mongardi--Wandel] \label{prop:aut0 computation for OG6}
		Let $\pi : X \to B$ be a Lagrangian fibration of an OG6-type hyper-K\"ahler manifold, obtained by the moduli of sheaves construction. Then
		\[ \Aut^{\circ}(X/B) \cong (\ZZ/2)^{\oplus 4} .\]
	\end{proposition}
	
	We note that the latter computation for OG6-type was already done by Mongardi--Wandel in \cite[\S 5]{mon-wan17}, as an intermediate step for their computation of the larger group $\Aut^{\circ}(X) \cong (\ZZ/2)^{\oplus 8}$. Thus proving \Cref{prop:aut0 computation for Kum} will be enough to conclude \Cref{thm:aut0 computation}. As mentioned before, the proof will be done by an explicit computation. In fact, the computation can be carried out further and calculates the polarization scheme $K_0$ as well. This is the second main result of this section.
	
	\begin{proposition} \label{prop:polarization scheme inclusion Kum}
		Let $\pi : X \to B$ be a Lagrangian fibration of a $\Kum_n$-type hyper-K\"ahler manifold, obtained by the moduli of sheaves construction from a triple $(S, l, s)$ in \Cref{def:moduli construction Kum}. Let $(d_1, d_2)$ be the polarization type of the ample class $l$.
		\begin{enumerate}
			\item If $n = 3$ and $d_1 = d_2 = 2$, then
			\[ K_0 \hookrightarrow \Aut^{\circ}(X/B) \hookrightarrow K_0[2] .\]
			\item Otherwise, we have
			\[ K_0 = \Aut^{\circ} (X/B) .\]
		\end{enumerate}
	\end{proposition}
	
	Contrary to the previous sections, all the discussions in this section will be algebraic. In particular, algebraic Chern classes and Chow groups will be used. Given a coherent sheaf $E$ on a smooth projective variety $S$, we denote by
	\[ c_i (E) \in H^{2i} (S, \ZZ), \qquad \tilde c_i (E) \in \CH^i (S) \]
	the $i$-th numerical (i.e., cohomological) and algebraic Chern classes of $E$, respectively.

	\subsection{Moduli of coherent sheaves on an abelian variety} \label{subsec:moduli construction Kum}
	In this subsection, we recall the construction of $\Kum_n$-type hyper-K\"ahler manifolds obtained from certain moduli spaces of sheaves on abelian varieties. We will mostly follow \cite{yos01}.
	
	Let $S$ be an abelian surface and $l \in \NS(S)$ an ample cohomology class with $\int_S l^2 = 2n+2$. Fix a nonzero class $s \in H^4 (S, \ZZ)$ so that we have a \emph{primitive} Mukai vector
	\begin{equation} \label{eq:mukai_vector}
		v = (0, l, s) \ \ \in H^*_{\operatorname{even}} (S, \ZZ) .
	\end{equation}
	Then the moduli space $M$ of stable sheaves on $S$ with Chern character $v$, with respect to a $v$-generic ample line bundle, becomes a smooth projective holomorphic symplectic variety of dimension $\langle v, v \rangle + 2 = 2n+4$. Denote by $\Pic^l_S$ a connected component of the Picard scheme of $S$ with numerical first Chern class $l$. Yoshioka proved the Albanese variety of $M$ is isomorphic to $S \times \Pic^l_S$, so that we can define the Albanese morphism $\Alb : M \to S \times \Pic^l_S$.
	
	To be more precise, we first need to choose a specific reference line bundle $L_0$ and coherent sheaf $E_0$ on $S$. We choose a line bundle $L_0$ a \emph{symmetric} ample line bundle in $\Pic_S^l$ (there are precisely $16$ of them). Fix a smooth curve $i : C_0 \hookrightarrow S$ in the linear system $|L_0|$ and define a reference coherent sheaf by $E_0 = i_* D$ for a line bundle $D$ on $C_0$ with degree $s+n+1$. The Riemann--Roch computation gives $\ch (E_0) = v$ and $\tilde c_1 (E_0) = \tilde c_1 (L_0)$. Say $\Sigma : \CH^2(S) \to S(\CC)$ is the summation map. The composition
	\[ \Sigma \circ i_* \circ \tilde c_1 : \Pic_{C_0}^{s+n+1} (\CC) \to \CH^1(C_0) \to \CH^2(S) \to S(\CC) \]
	is surjective by \Cref{lem:albanese map description}. Due to this fact, we may choose an appropriate line bundle $D$ on $C_0$ to further assume $\Sigma \big( \tilde c_2 (E_0) \big) = 0$ (this again uses Riemann--Roch). Once choosing these reference points, the Albanese morphism can be explicitly described by
	\begin{equation} \label{eq:albanese morphism Kum}
		\Alb : M \to S \times \Pic^l_S, \qquad [E] \mapsto \big( \mathfrak c(E), \ \tilde c_1(E) \big) ,
	\end{equation}
	where we define $\mathfrak c(E) = \Sigma \big( \tilde c_2 (E) \big)$. It sends the reference point $[E_0]$ to the origin $(0, [L_0])$ of $S \times \Pic_S^l$. The morphism becomes an \'etale trivial fiber bundle with $\Kum_n$-type projective hyper-K\"ahler manifold fibers. We will work with the central fiber
	\[ X = \Alb^{-1} \big( 0, [L_0] \big) .\]
	
	Due to our choice of the Mukai vector $v$ in \eqref{eq:mukai_vector}, the above construction further comes with a Lagrangian fibration. Consider a connected component $\tilde B$ of the Chow variety of effective divisors on $S$ with numerical first Chern class $l$. Le Potier \cite{lepot93} constructed a morphism
	\[ \Supp : M \to \tilde B, \qquad [E] \mapsto [\Fitt_0 E] ,\]
	where $\Fitt_0 E$ is the Fitting support of a coherent sheaf $E$. Finally, consider the Poincar\'e line bundle $\mathcal P$ on $S \times \Pic^l_S$, the universal family of line bundles with the numerical Chern class $l$. Denote by $r : S \times \Pic^l_S \to \Pic^l_S$ the second projection. Then by Riemann--Roch, $r_* \mathcal P$ is a vector bundle of rank $n+1$. Its projectivization is a Zariski locally trivial $\PP^n$-bundle
	\[ \operatorname{LB} : \tilde B \to \Pic^l_S, \qquad [C] \mapsto \mathcal [\mathcal O_S(C)] .\]
	
	Gathering all the morphisms together, one easily checks we have a commutative diagram
	\[\begin{tikzcd}[sep=tiny]
		M \arrow[dd, "\Alb"'] \arrow[rd, "{(\mathfrak c, \, \Supp)}"] \\
		& S \times \tilde B \arrow[ld, "\id \times \operatorname{LB}"] \\
		S \times \Pic^l_S
	\end{tikzcd}.\]
	This is an isotrivial family of Lagrangian fibered hyper-K\"ahler manifolds in the sense of \Cref{def:family of Lagrangian fibration}. Setting $B = \operatorname{LB}^{-1}([L_0]) = |L_0| \cong \PP^n$, we obtain a Lagrangian fibration $\pi : X \to B$.
	
	\begin{definition} \label{def:moduli construction Kum}
		Let $(S, l)$ be a degree $n+1$ polarized abelian surface with a polarization type $(d_1, d_2)$ and $s \in H^4 (S, \ZZ)$ any nonzero class with $\gcd(d_1, s) = 1$. Then the above construction $\pi : X \to B$ is called the \emph{moduli construction of $\Kum_n$-type} associated to the triple $(S, l, s)$. It is a Lagrangian fibration of a projective hyper-K\"ahler manifold of $\Kum_n$-type to a projective space.
	\end{definition}

	\subsection{Describing the $H^2$-trivial automorphisms of $\Kum_n$-type moduli constructions}
	Recall that \cite{boi-nei-sar11} and \cite{has-tsc13} proved that any $\Kum_n$-type hyper-K\"ahler manifold have the group of $H^2$-trivial automorphisms
	\[ \Aut^{\circ}(X) \cong \ZZ/2 \ltimes (\ZZ/n+1)^{\oplus 4} .\]
	The goal of this subsection is to explicitly describe such automorphisms for the moduli construction. Note that describing such automorphisms is about $X$ itself but not about the Lagrangian fibration $\pi : X \to B$. Hence, the Lagrangian fibration plays no role in this subsection.
	
	Recall that we have fixed the origin $[L_0] \in \Pic^l_S$, a symmetric ample line bundle on $S$. By the general theory of abelian varieties, there exists a dual ample line bundle $\check L_0$ on the dual abelian variety $\check S$ (see \cite[\S 14.4]{bir-lan:abelian}). The ample line bundles $L_0$ and $\check L_0$ induce polarization isogenies
	\[ \varphi : S \to \check S, \qquad \check \varphi : \check S \to S ,\]
	making their compositions the multiplication endomorphisms
	\begin{equation} \label{eq:isogeny composition}
		[n+1] : S \xrightarrow{\varphi} \check S \xrightarrow{\check \varphi} S, \qquad [n+1] : \check S \xrightarrow{\check \varphi} S \xrightarrow{\varphi} \check S .
	\end{equation}
	Since $L_0$ has a polarization type $(d_1, d_2)$, the dual line bundle $\check L_0$ has a polarization type $(d_1, d_2)$ as well. In particular, we have an isomorphism
	\begin{equation} \label{eq:kernel}
		\ker \check \varphi \cong (\ZZ/d_1 \oplus \ZZ/d_2)^{\oplus 2} .
	\end{equation}
	
	A closed point $x$ on $S$ defines a translation automorphism by $x$. Our notation for the translation automorphism is
	\[ t_x : S \to S, \qquad y \mapsto y + x .\]
	A closed point $\xi$ on $\check S$ represents a numerically trivial line bundle on $S$. Considering $\xi$ both as a closed point on $\check S$ and a line bundle on $S$ can possibly lead to a confusion. Thus, we will write
	\[ P_{\xi} : \mbox{ numerically trivial line bundle on } S \mbox{ corresponding to } \xi \in \check S .\]
	With these notation in mind, we can explicitly realize the $\Aut^{\circ}(X)$-action for the moduli of sheaves construction $X$.
	
	\begin{proposition} \label{prop:aut0 computation for Kum2}
		Let $X$ be a $\Kum_n$-type moduli construction associated to a triple $(S, l, s)$ in \Cref{def:moduli construction Kum}. Then
		\begin{enumerate}
			\item We have an isomorphism
			\[ \Aut^{\circ} (X) = \{ \pm 1 \} \ltimes \{ (x, \xi) \in S[n+1] \times \check S[n+1] : \ \varphi (x) = 0, \ \ \check \varphi (\xi) = sx \} .\]
			\item With the above identification, the $\Aut^{\circ}(X)$-action on $X$ is defined by
			\[ (1, x, \xi).[E] = [t_x^* E \otimes P_{\xi}], \qquad (-1, x, \xi).[E] = [t_x^* ([-1]^* E) \otimes P_{\xi}] ,\]
			where $[-1] : S \to S$ is the multiplication by $-1$ automorphism on $S$.
		\end{enumerate}
	\end{proposition}
	
	The rest of this subsection is devoted to the proof of \Cref{prop:aut0 computation for Kum2}. To start, we note that Yoshioka has already computed an explicit trivialization of Albanese morphism $\Alb : M \to S \times \Pic^l_S$. Yoshioka's trivialization is obtained by the base change $[n+1] : S \times \Pic^l_S \to S \times \Pic^l_S$, which is a degree $(n+1)^8$ isogeny. As we will see in a moment, this is not a minimal isogeny in the sense of \Cref{def:minimal isogeny}. Using the methods in \Cref{sec:minimal split covering}, we first prove the morphism
	\begin{equation} \label{eq:minimal isogeny}
		\phi : S \times \Pic^l_S \to S \times \Pic^l_S, \qquad (y, [L]) \mapsto (sy - \check \varphi(L \otimes L_0^{-1}), \ [L_0 \otimes P_{\varphi(y)}])
	\end{equation}
	is the minimal isogeny trivializing the Albanese morphism.
	
	\begin{proposition} \label{prop:minimal isogeny for moduli construction}
		The base change \eqref{eq:minimal isogeny} is the minimal isogeny trivializing the Albanese morphism $\Alb : M \to S \times \Pic^l_S$ in the sense of \Cref{def:minimal isogeny}.
	\end{proposition}
	\begin{proof}
		Start from Yoshioka's diagram \cite[\S 4.1]{yos01} trivializing the Albanese morphism, which is a cartesian diagram
		\begin{equation} \label{diag:yoshioka}
		\begin{tikzcd}
			X \times (S \times \Pic^l_S) \arrow[r, "\Phi'"] \arrow[d, "\pr_2"] & M \arrow[d, "\Alb"] \\
			S \times \Pic^l_S \arrow[r, "{[n+1]}"] & S \times \Pic^l_S
		\end{tikzcd} .
		\end{equation}
		Here $\Phi' : X \times (S \times \Pic^l_S) \to M$ is defined to be a Galois \'etale morphism
		\[ \Phi' ([E], y, [L]) = \Big[ t_{\check \varphi (L \otimes L_0^{-1})}^* E \otimes (L \otimes L_0^{-1})^{\otimes s} \otimes P_{-\varphi(y)} \Big] .\]
		Note that our convention differs by sign to Yoshioka's original paper, because Yoshioka's dual line bundle $\check L_0$ differs to ours by sign.
		
		The Galois group of the base change $[n+1]$ is the group of $(n+1)$-torsion points $S[n+1] \times \check S[n+1] \cong (\ZZ/n+1)^{\oplus 8}$. By \Cref{prop:galois}, it acts on $X \times (S \times \Pic^l_S)$ by translation on the second factor
		\[ (x, \xi) . ([E], y, [L]) = ([E], y+x, [L \otimes P_{\xi}]) .\]
		One computes the descent of this action to $M$ via $\Phi'$:
		\[ (x,\xi).[E] = \Big[ t_{\check \varphi (\xi)}^* E \otimes P_{s \xi - \varphi(x)} \Big] .\]
		This is the $S[n+1] \times \check S[n+1]$-action on $X = \Alb^{-1} (0, [L_0])$ in \Cref{prop:galois}. One sees this action is not an effective action, and the kernel of the action is precisely
		\[ \{ (x, \xi) \in S[n+1] \times \check S[n+1] : \ \check \varphi(\xi) = 0, \ s\xi - \varphi(x) = 0 \} .\]
		
		To kill the kernel and obtain an effective action, take a Galois quotient corresponding to the kernel (via the Galois correspondence). This is an isogeny $\psi : S \times \Pic^l_S \to S \times \Pic^l_S$ defined by
		\[ \psi(y, [L]) = (\check \varphi (L \otimes L_0^{-1}), \ (L \otimes L_0^{-1})^{\otimes s} \otimes P_{-\varphi(y)}) .\]
		One can check the morphism $\phi$ in \eqref{eq:minimal isogeny} is precisely the isogeny making $\phi \circ \psi = [n+1]$ (here one needs to use \eqref{eq:isogeny composition}, but we omit the computation). The result is a factorization of \eqref{diag:yoshioka} into the minimal isogeny
		\[\begin{tikzcd}
			X \times (S \times \Pic^l_S) \arrow[r, "\Psi"] \arrow[d, "\pr_2"] & X \times (S \times \Pic^l_S) \arrow[r, "\Phi"] \arrow[d, "\pr_2"] & M \arrow[d, "\Alb"] \\
			S \times \Pic^l_S \arrow[r, "\psi"] & S \times \Pic^l_S \arrow[r, "\phi"] & S \times \Pic^l_S
		\end{tikzcd}.\]
		Here our new morphism $\Phi$, Beauville's minimal split covering of $M$, turns out to have a better form than the original $\Phi'$:
		\begin{equation} \label{eq:Phi}
			\Phi([E], y, [L]) = \Big[ t_y^* E \otimes (L \otimes L_0^{-1}) \Big] .
		\end{equation}
		The claim follows.
	\end{proof}
	
	Again thanks to \Cref{prop:galois}, we have a canonical, effective and $H^2$-trivial $\Gal(\phi)$-action on $X$. The Galois group $\Gal(\phi)$ is captured by the kernel of $\phi$, so we have
	\begin{equation} \label{eq:galois_group}
		\Gal(\phi) = \{ (x, \xi) \in S[n+1] \times \check S[n+1] : \ \varphi (x) = 0, \ \check \varphi (\xi) = sx \} .
	\end{equation}
	This explains the isomorphism in \Cref{prop:aut0 computation for Kum2}. The $\Gal(\phi)$-action on the fiber $X$ is obtained via the description of $\Phi$ in \eqref{eq:Phi}. This explains how we obtained the group action in \Cref{prop:aut0 computation for Kum2}.
	
	We can compute $\Gal(\phi)$ more explicitly.
	
	\begin{lemma} \label{lem:galois computation}
		$\Gal(\phi) \cong (\ZZ/n+1)^{\oplus 4}$.
	\end{lemma}
	\begin{proof}
		Let us compute the group \eqref{eq:galois_group} explicitly. The expression involves the abelian surfaces $S$ and its dual $\check S$, their $(n+1)$-torsion points and their polarization isogenies $\varphi$ and $\check \varphi$. Therefore, the expression is independent on the complex structure on $S$ and the question is topological. We may fix polarization bases $H_1(S, \ZZ) = \ZZ \{ e_1, \cdots, e_4 \}$ and $H_1(\check S, \ZZ) = \ZZ \{ e_1^*, \cdots, e_4^* \}$ so that we can identify $S = (\RR/\ZZ) \{ e_1, \cdots, e_4 \}$ and $\check S = (\RR/\ZZ)\{ e_1^*, \cdots, e_4^* \}$. The polarization isogenies with respect to them are
		\begin{equation} \label{eq:varphi matrix form}
			\varphi = \begin{psmallmatrix}
				0 & 0 & d_1 & 0 \\
				0 & 0 & 0 & d_2 \\
				-d_1 & 0 & 0 & 0 \\
				0 & -d_2 & 0 & 0
			\end{psmallmatrix}, \qquad \check \varphi = \begin{psmallmatrix}
				0 & 0 & -d_2 & 0 \\
				0 & 0 & 0 & -d_1 \\
				d_2 & 0 & 0 & 0 \\
				0 & d_1 & 0 & 0
			\end{psmallmatrix} .
		\end{equation}
		
		Writing the coordinate $(a_1, \cdots, a_4)$ for $S = (\RR/\ZZ)^4$ and $(b_1, \cdots, b_4)$ for $\check S = (\RR/\ZZ)^4$, one explicitly computes
		\[ \Gal(\phi) = \{ (a_i, b_i)_{i=1}^4 \in \Big( \tfrac{1}{n+1}\ZZ / \ZZ \Big)^{\oplus 8} : \ d_1 a_1 = 0, \ d_2 b_3 = sa_1, \ \cdots \} \cong A^{\oplus 4} ,\]
		where the abelian group $A$ is defined by
		\[ A = \{ (a, b) \in (\ZZ/n+1)^{\oplus 2} : \ d_1 a = 0, \ d_2 b = sa \} .\]
		Notice that $\gcd(d_1, s) = 1$ by the very assumption we had in \Cref{def:moduli construction Kum}. Now $A \cong \ZZ/n+1$ by the following simple computational lemma, and the desired isomorphism is proved.
	\end{proof}
	
	\begin{lemma}
		Let $p, q, s$ be nonzero integers. Set $m = pq$ and assume either $\gcd(p,s) = 1$ or $\gcd(q,s) = 1$. Then the abelian group
		\[ A = \{ (a,b) \in (\ZZ/m)^{\oplus 2} : \ pa = 0, \ qb = sa \} \]
		is isomorphic to $\ZZ/m$.
	\end{lemma}
	\begin{proof}
		The group $A$ is realized by the kernel of a homomorphism $f : (\ZZ/m)^{\oplus 2} \to (\ZZ/m)^{\oplus 2}$, $f = \begin{psmallmatrix} p & 0 \\ -s & q \end{psmallmatrix}$. Adjusting the bases of both the domain and codomain (i.e., performing elementary row and column operations), the matrix can be transformed into $\begin{psmallmatrix} 1 & 0 \\ 0 & pq \end{psmallmatrix} = \begin{psmallmatrix} 1 & 0 \\ 0 & 0 \end{psmallmatrix}$. Here one needs the assumption $\gcd(p, s) = 1$ or $\gcd(q, s) = 1$ to apply the Euclidean algorithm. The claim follows.
	\end{proof}
	
	We have described $\Gal(\phi) \cong (\ZZ/n+1)^{\oplus 4}$-action on $X$ acting trivially on $H^2$. Since $\Aut^{\circ}(X) \cong \ZZ/2 \ltimes (\ZZ/n+1)^{\oplus 4}$, we still need an additional $\ZZ/2$-part to describe. Fortunately, this is not hard to guess. Construct an involution $\iota$ on $X \times (S \times \Pic^l_S)$ by
	\[ \iota ([E], y, [L]) = (\big[ [-1]^*E \big], -y, \big[ [-1]^* L \big] ) .\]
	Because we are not relying on the general theory anymore, we need to check $\iota$ acts on $M$. We omit the typical Chern class computation.
	
	The involution does not commute with the $S \times \check S$-action on $X \times (S \times \Pic^l_S)$, and this is the reason why $\ZZ/2$ should act on $(\ZZ/n+1)^{\oplus 4}$ nontrivially and leads to the semi-direct product. The action descends to $M$ as a satisfying form
	\[ \iota ([E]) = [[-1]^* E] .\]
	To check $\iota$ acts on the fiber $X = \Alb^{-1} (0, [L_0])$, we need to check $\mathfrak c ([-1]^* E) = 0$ and $\tilde c_1 ([-1]^*E) = \tilde c_1 (L_0)$ for all $[E] \in X$. The former follows from definition and the latter follows from the fact that $L_0$ is a \emph{symmetric} line bundle. It remains to prove $\iota$ acts on the second cohomology of $X$ as the identity. We have already proved in \Cref{prop:galois} that $H^2 (M, \QQ) \to H^2 (X, \QQ)$ is surjective. Hence we only need to prove $\iota$ acts on $H^2 (M, \QQ)$ as the identity. This follows because $\iota$ is induced from the automorphism $[-1]$ on $S$, $[-1]$ acts on $H^2 (S, \QQ)$ trivially and finally the Hodge structure $H^2 (M, \QQ)$ is obtained by a tensor construction of $H^2 (S, \QQ)$ by \cite{bul20}. This exhausts the entire $\Aut^{\circ} (X)$-action description on $X$ and hence completes the proof of \Cref{prop:aut0 computation for Kum2}.

	\subsection{Automorphisms respecting the Lagrangian fibration} \label{subsec:aut0 description}
	With \Cref{prop:aut0 computation for Kum2} at hand, the proof of \Cref{prop:aut0 computation for Kum} becomes fairly straightforward. Any $H^2$-trivial automorphism is of the form
	\[ f = (\pm 1, x, \xi) \qquad \mbox{for}\quad x \in \ker \varphi, \quad \xi \in \check S[n+1] \quad \mbox{with} \quad \check \varphi (\xi) = sx .\]
	
	Let us first consider automorphisms of the form $f = (1, x, \xi)$. It acts on $X$ by $f : [E] \mapsto [t_x^*E \otimes P_{\xi}]$. Recall $\pi : X \to B$ was by definition the (Fitting) support map $\Supp : [E] \mapsto [\Fitt_0 E]$. The support of $t_x^* E \otimes P_{\xi}$ is $\Supp E - x$, so $f$ respects $\pi$ if and only if $\Supp E = \Supp E - x$. This means $x = 0$ and $\xi \in \ker \check \varphi$. Therefore, such automorphisms form a group $\ker \check \varphi$, which is isomorphic to $(\ZZ/d_1 \oplus \ZZ/d_2)^{\oplus 2}$ by \eqref{eq:kernel}. 
	
	We next consider automorphisms of the form $f = (-1, x, \xi)$. A similar argument shows $f$ respects $\pi$ if and only if $\Supp E = [-1]^* \Supp E - x$ for all $[E] \in X$. In other words, we have $D = [-1]^* D - x$ for all $D \in |L_0|$. Fix any $\frac{1}{2}x \in S$ with $2 \cdot \left( \frac{1}{2} x \right) = x$. Then this condition is equivalent to every $D \in |t_{-\frac{1}{2}x}^* L_0|$ being a symmetric divisor. In particular, $t_{-\frac{1}{2}x}^* L_0$ is a symmetric line bundle. We have chosen $L_0$ to be a symmetric line bundle, so this implies $\frac{1}{2} x$ is a $2$-torsion point, or $x = 0$. The condition now becomes that every $D \in |L_0|$ is symmetric.
	
	\begin{lemma}
		Let $S$ be an abelian surface and $L_0$ a symmetric ample line bundle on it. Then every divisor in the complete linear system $|L_0|$ is symmetric if and only if $L_0$ has a polarization type $(1, 1)$, $(1, 2)$ or $(2, 2)$.
	\end{lemma}
	\begin{proof}
		Assume $L_0$ has one of the three given polarization types. When $L_0$ is a principal polarization, $|L_0|$ consists of a single symmetric divisor. When $L_0$ is twice a principal polarization, the statement is proved in \cite[Thm 4.8.1]{bir-lan:abelian}. When $L_0$ has a polarization type $(1,2)$, the statement can be found in \cite[Prop 1.6]{barth87}.
		
		Conversely, let us assume every divisor in $|L_0|$ is symmetric. Denote by $H^0 (S, L_0)_{\pm}$ the $\pm 1$-eigenspaces of the involution $[-1]^*$ on $H^0 (S, L_0)$. Every divisor in $|L_0|$ is symmetric if and only if either $H^0 (S, L_0)_+ = 0$ or $H^0 (S, L_0)_- = 0$. The dimensions of $H^0 (S, L_0)_{\pm}$ are computed in \cite[Ex 4.12.11]{bir-lan:abelian}: if we let the polarization type of $L_0$ to be $(d_1, d_2)$ then
		\[ h^0 (L_0)_{\pm} = \tfrac{1}{2} h^0 (L_0), \quad \tfrac{1}{2} h^0 (L_0) \pm 2^{1-s} \quad \mbox{or} \quad \frac{1}{2} h^0 (L_0) \mp 2^{1-s} ,\]
		where $0 \le s \le 2$ is an integer where $d_1, \cdots, d_s$ are odd and $d_{s+1}$ is even. There are three possibilities making $h^0 (L_0)_+ = 0$ or $h^0 (L_0)_- = 0$:
		\begin{enumerate}
			\item $h^0 (L_0) = 1$ and $s = 2$;
			\item $h^0 (L_0) = 2$ and $s = 1$; or
			\item $h^0 (L_0) = 4$ and $s = 0$.
		\end{enumerate}
		Using $h^0 (L_0) = d_1 d_2$, it is easy to check these are the desired three cases in the statement.
	\end{proof}
	
	From the lemma, there are only three possible polarization types of $L_0$. We have assumed from the very beginning that $\int_S l^2 = 2n+2 = 2d_1 d_2$ and $n \ge 2$. The first two cases are thus excluded. The only possible case is when $n = 3$ and $d_1 = d_2 = 2$. This completes the proof of \Cref{prop:aut0 computation for Kum}.

	\subsection{The polarization scheme of generalized Kummer varieties} \label{subsec:polarization scheme Kum}
	This subsection will be devoted to the proof of \Cref{prop:polarization scheme inclusion Kum}. Let us keep assume $\pi : X \to B$ is a $\Kum_n$-type moduli construction. The computations in this subsection are highly influenced by \cite[\S 6]{wie18}. Recall from \S \ref{subsec:moduli construction Kum} that we had a Fitting support morphism $\Supp : M \to \tilde B$ over $\Pic^l_S$. Fix a point $[L_0] \in \Pic^l_S$ and consider the fibers of $M$ and $\tilde B$ over it. We obtain a morphism
	\[ \Supp : Y \to B ,\]
	where $B = |L_0|$ is a complete linear system and $Y \subset M$ consists of torsion coherent sheaves $E$ on $S$ with $\ch(E) = v$ and $\tilde c_1(E) = \tilde c_1(L_0)$. The $\Kum_n$-type hyper-K\"ahler manifold $X$ is obtained by a fiber of the isotrivial fiber bundle $\mathfrak c : Y \to S$.
	
	Consider the universal family $\mathcal C \to B$ of curves on $S$ parametrizing effective divisors in $B = |L_0|$. Since $L_0$ is ample, by Bertini there exists a Zariski dense open subset $B_0 \subset B$ parametrizing smooth curves. The restriction of the universal family $\mathcal C_0 \to B_0$ becomes a smooth projective family of curves. The following lemma is standard and we omit its proof.
	
	\begin{lemma}
		The morphism $\Supp : Y_0 = \Supp^{-1}(B_0) \to B_0$ is isomorphic to the relative Picard scheme of the universal family of curves $\Pic^d_{\mathcal C_0/B_0} \to B_0$ for $d = s + n + 1$.
	\end{lemma}
	
	The lemma in particular says $Y_0 \to B_0$ is a torsor under the numerically trivial relative Picard scheme
	\[ J_0 = \Pic^0_{\mathcal C_0/B_0} \to B_0 .\]
	Since $\mathcal C_0 / B_0$ is a smooth projective family of curves, its relative Picard scheme $J_0$ is a canonically principally polarized abelian scheme. As standard, we will call it a relative Jacobian of the family and identify $\check J_0 = J_0$. Notice that we now have four different spaces over $B_0$: $P_0$, $X_0$, $J_0$ and $Y_0$. The space $X_0$ is a $P_0$-torsor as usual, and we also have $Y_0$ as a $J_0$-torsor. Since $P_0$ and $J_0$ are translation automorphism schemes of $X_0 \subset Y_0$, we have an inclusion $P_0 \subset J_0$. Our first goal is to describe the quotient of this inclusion.
	
	\begin{proposition}
		There exists a short exact sequence of abelian schemes over $B_0$:
		\begin{equation} \label{eq:ses to jacobian}
			\begin{tikzcd}
				0 \arrow[r] & P_0 \arrow[r] & J_0 \arrow[r] & S \times B_0 \arrow[r] & 0
			\end{tikzcd}.
		\end{equation}
	\end{proposition}
	\begin{proof}
		The universal family $\mathcal C_0 \to B_0$ is a subvariety of the product $i : \mathcal C_0 \hookrightarrow S \times B_0$. This induces a pullback morphism $i^* : \check S \times B_0 \to J_0$ between their relative Picard schemes over $B_0$. The morphism $J_0 \to S \times B_0$ can be constructed by the dual of it. Fiberwise, it is the morphism $J_C \to S$ induced by the universal property of the Albanese morphism applied to $i : C \hookrightarrow S$.
		
		We prove the kernel of the morphsim $J_0 \to S \times B_0$ is $P_0$. The claim can be verified fiberwise. Fix a closed point $[C] \in B_0$ corresponding to a smooth curve $i : C \hookrightarrow S$. Over it, a closed point of $Y_0$ (resp., $J_0$) is represented by a degree $d$ line bundle $L$ on $C$ (resp., degree $0$ line bundle $M$ on $C$). The $J_0$-action on $Y_0$ is given by $[i_* M] . [i_* L] = [i_* (L \otimes M)]$. Recall that $X$ is a fiber of the morphism $\mathfrak c : Y \to S$. Hence the abelian scheme $P_0$ consists of translation automorphisms of $J_0$ invariant under the morphism $\mathfrak c$. Recall the definition of $\mathfrak c$ in \eqref{eq:albanese morphism Kum}. A Riemann--Roch computation gives us
		\[ \mathfrak c ([i_* (L \otimes M)]) = \mathfrak c ([i_*L]) - \Sigma i_* \tilde c_1 (M) ,\]
		where $\tilde c_1(M) \in \CH^1(S)$, $i_* : \CH^1 (S) \to \CH^2(S)$ and $\Sigma : \CH^2(S) \to S(\CC)$ is a summation map. This proves the $[i_* M]$-action on the fiber of $Y_0$ is $\mathfrak c$-invariant if and only if $\Sigma i_* \tilde c_1(M) = 0$. The claim follows by the following lemma, which is already proved in \cite[(6.8)]{wie18}.
	\end{proof}
	
	\begin{lemma} \label{lem:albanese map description}
		The morphism $J_C \to S$ sends a closed point $[M] \in J_C(\CC)$ to $\Sigma i_* \tilde c_1(M) \in S(\CC)$.
	\end{lemma}
	
	The dual of \eqref{eq:ses to jacobian} is automatically (e.g., \cite[Prop 2.4.2]{bir-lan:abelian}) a short exact sequence of abelian schemes
	\begin{equation} \label{eq:ses to jacobian dual}
		\begin{tikzcd}
			0 \arrow[r] & \check S \times B_0 \arrow[r] & J_0 \arrow[r] & \check P_0 \arrow[r] & 0
		\end{tikzcd}.
	\end{equation}
	In particular, $P_0$ and $\check S \times B_0$ are both abealin subschemes of a bigger abelian scheme $J_0$. The following proposition describes the polarization scheme $K_0$ more explicitly for the moduli constructions.
	
	\begin{proposition} \label{prop:polarization scheme equality moduli construction}
		We have the following two additional descriptions of the polarization scheme $K_0$ as a $B_0$-group scheme:
		\[ K_0 = P_0 \cap (\check S \times B_0) = \ker (\check \varphi \times \id : \check S \times B_0 \to S \times B_0) .\]
	\end{proposition}
	\begin{proof}
		Fiberwise at a closed point $[C] \in B_0$, the sequences \eqref{eq:ses to jacobian} and \eqref{eq:ses to jacobian dual} are short exact sequences of abelian varieties
		\[\begin{tikzcd}
			0 \arrow[r] & F \arrow[r] & J_C \arrow[r] & S \arrow[r] & 0
		\end{tikzcd}, \qquad \begin{tikzcd}
			0 \arrow[r] & \check S \arrow[r] & J_C \arrow[r] & \check F \arrow[r] & 0
		\end{tikzcd}.\]
		Here $F = \nu^{-1}([C])$ is a fiber of $P_0$ and $J_C$ is the Jacobian of the curve $C$. The two abelian subvarieties $F$ and $\check S$ of the principally polarized abelian variety $J_C$ are the so-called complementary abelian subvarieties (see \cite[\S 6.4]{wie18} and \cite[\S 12.1]{bir-lan:abelian}). In this case, we have an equality (\cite[Cor 12.1.4]{bir-lan:abelian})
		\[ \ker (F \to J_C \to \check F) = F \cap \check S = \ker (\check S \to J_C \to S) .\]
		We will soon prove in \Cref{lem:composition is polarization isogeny} that the composition $\check S \to J_C \to S$ is precisely the polarization isogeny $\check \varphi$ regardless of the choice of a closed point $[C] \in B_0$. Given this, we obtain a sequence of identities of group schemes
		\[ \ker (P_0 \to J_0 \to \check P_0) = P_0 \cap (\check S \times B_0) = \ker (\check \varphi \times \id : \check S \times B_0 \to S \times B_0) .\]
		From the last description and \eqref{eq:kernel}, this group scheme is a constant group scheme with fibers $(\ZZ/d_1 \oplus \ZZ/d_2)^{\oplus 2}$. The first description is describing the polarization scheme $K_0$; combine the uniqueness of the polarization in \Cref{thm:abelian scheme} and the computation of polarization types in \Cref{thm:polarization type computation}. The claim follows.
	\end{proof}
	
	\begin{lemma} \label{lem:composition is polarization isogeny}
		The composition $\check S \to J_C \to S$ is the polarization isogeny $\check \varphi$.
	\end{lemma}
	\begin{proof}
		Denote by $i : C \hookrightarrow S$ the closed immersion. At the level of first homologies, the composition $\check S \to J_C \to S$ becomes a Hodge structure homomorphism
		\[ H_1 (\check S, \ZZ) = H^1 (S, \ZZ) \xrightarrow{i^*} H^1 (C, \ZZ) \xrightarrow{i_*} H^3 (S, \ZZ) = H_1 (S, \ZZ) .\]
		Hence the composition is $i_* \circ i^*$, which is the multiplication map by $c_1 (\mathcal O_S(C)) \in H^2 (S, \ZZ)$. Because we have chosen $[C]$ in a complete linear system $|L_0|$, it is a multiplication by $c_1 (L_0) = l$.
		
		Therefore, the question reduces to the following claim: the dual polarization $\check \varphi : \check S \to S$ is given by $l \cup - : H^1(S, \ZZ) \to H^3(S, \ZZ)$. Again choose polarization bases $H_1(S, \ZZ) = \ZZ \{ e_1, \cdots, e_4 \}$ and $H_1(\check S, \ZZ) = H^1 (S, \ZZ) = \ZZ \{ e_1^*, \cdots, e_4^* \}$ as in \Cref{lem:galois computation}. The polarization isogenies $\varphi$ and $\check \varphi$ have the matrix forms \eqref{eq:varphi matrix form}. The ample class $l$ is the skew-symmetric bilinear map $\varphi : H_1 (S, \ZZ) \otimes H_1 (S, \ZZ) \to \ZZ$ considered as an element of $H^2 (S, \ZZ)$. Hence it is $l = d_1 e_1^* \wedge e_3^* + d_2 e_2^* \wedge e_4^*$.
		
		We can now explicitly compute the map $l \cup - : H^1 (S, \ZZ) \to H^3 (S, \ZZ)$:
		\[\begin{aligned}
			& e_1^* \mapsto d_2 e_1^* \wedge e_2^* \wedge e_4^* , \qquad && e_2^* \mapsto -d_1 e_1^* \wedge e_2^* \wedge e_3^*, \\
			& e_3^* \mapsto -d_2 e_2^* \wedge e_3^* \wedge e_4^* , \qquad && e_4^* \mapsto d_1 e_1^* \wedge e_3^* \wedge e_4^* .
		\end{aligned}\]
		The Poincar\'e duality $H_1 (S, \ZZ) = H^3 (S, \ZZ)$ yields the basis of $H^3 (S, \ZZ)$:
		\[ \{ e_2^* \wedge e_3^* \wedge e_4^*, \ \ -e_1^* \wedge e_3^* \wedge e_4^*, \ \ e_1^* \wedge e_2^* \wedge e_4^*, \ \ -e_1^* \wedge e_2^* \wedge e_3^* \} .\]
		With respect to it, the matrix form of the multiplication coincides with precisely the matrix form of $\check \varphi$ above. (Compare this lemma with \cite[Lem 6.14]{wie18}.)
	\end{proof}
	
	\begin{proof}[Proof of \Cref{prop:polarization scheme inclusion Kum}]
		Recall from \S \ref{subsec:aut0 description} the complete description of $\Aut^{\circ}(X/B)$. Let us assume $n \neq 3$ or $(d_1, d_2) \neq (2,2)$, so that every automorphism $f \in \Aut^{\circ}(X/B)$ is of the form $(1, 0, \xi)$ for $\xi \in \ker \check \varphi$. It acts on $Y$ by $[E] \mapsto [E \otimes P_{\xi}]$, where $P_{\xi}$ is the numerically trivial line bundle on $S$ represented by $\xi \in \ker \check \varphi \subset \check S$. On $Y_0$, closed points are of the form $[E] = [i_*L]$ where $L$ is a line bundle on a smooth curve $i : C \hookrightarrow S$. Hence $f$ acts on it by
		\[ f . [i_*L] = [i_*L \otimes P_{\xi}] = [i_* (L \otimes i^* P_{\xi})] .\]
		This means the global section of $J_0 \to B_0$ defined by $f$ represents a line bundle $[i^* P_{\xi}]$ over $[C] \in B_0$. The inclusion $\check S \times B_0 \subset J_0$ was by definition the pullback morphism of line bundles. Hence $f$ defines in fact a global section $\xi = [P_{\xi}] \in \check S$ of the constant group scheme $\check S \times B_0$. This coincides with the description of the polarization scheme $K_0$ in \Cref{prop:polarization scheme equality moduli construction}, proving the desired $K_0 = \Aut^{\circ}(X/B)$.
		
		The proof for the exceptional case $n = 3$ and $d_1 = d_2 = 2$ goes identical. The only difference is that the automorphisms $f \in \Aut^{\circ}(X/B)$ of the form $(1, 0, \xi)$ consist of an index $2$ subgroup of $\Aut^{\circ}(X/B)$. So this case proves $K_0 \subset \Aut^{\circ}(X/B)$ as an index $2$ subgroup. The second inclusion $\Aut^{\circ}(X/B) \subset K_0[2]$ follows from \Cref{prop:polarization scheme ver2} since we have $\operatorname{div}(h) = d_1 = 2$.
	\end{proof}

	\section{The dual Lagrangian fibration of a compact hyper-K\"ahler manifold} \label{sec:dual Lagrangian fibration}
	Combining the previous results, we can prove the polarization scheme extends to a constant subgroup scheme of $\Aut^{\circ}(X/B)$ over $B$ for known hyper-K\"ahler manifolds.
	
	\begin{theorem} \label{thm:polarization scheme equality}
		Let $\pi : X \to B$ be a Lagrangian fibration of a compact hyper-K\"ahler manifold of $\text{K3}^{[n]}$, $\Kum_n$, OG10 or OG6-type. Then the polarization scheme $K_0 \to B_0$ uniquely extends to a constant group scheme $K \to B$ that is a subgroup scheme of the constant group scheme $\Aut^{\circ}(X/B)$.
%		\[ K = \Aut^{\circ} (X/B) .\]
	\end{theorem}
	\begin{proof}
		When $X$ is of $\text{K3}^{[n]}$ or OG10-type, both the polarization scheme $K$ and the global sections defined by $\Aut^{\circ}(X/B)$ are the zero section of the abelian scheme $P_0$. Hence the claim is trivial. When $X$ is of OG6-type, lattice theory forces $\operatorname{div}(h) = 1$ as shown in \cite[Lem 7.1]{mon-rap21}. \Cref{prop:polarization scheme ver2} applies and we get an inclusion $\Aut^{\circ} (X/B) \hookrightarrow K_0$. Combining \Cref{thm:polarization type computation} and \ref{thm:aut0 computation}, the inclusion is forced to be an equality fiberwise. Hence we get the global equality $K_0 = \Aut^{\circ}(X/B)$. In particular, $K_0$ extends over $B$ to a constant group scheme $\Aut^{\circ}(X/B)$.
		
		Assume $X$ is of $\Kum_n$-type and the polarization type of $\pi$ is not $(1,2,2)$. In this case, \Cref{prop:polarization scheme inclusion Kum} together with \Cref{prop:polarization scheme ver1} implies an equality of group schemes $K_0 = \Aut^{\circ}(X/B)$. The remaining case is when $X$ is of $\Kum_3$-type and the polarization type of $\pi$ is $(1,2,2)$. In this case, we have $\operatorname{div}(h) = 2$ so \Cref{prop:polarization scheme ver1} guarantees $\Aut^{\circ} (X/B) \subset K_0[2]$, where $K_0[2] = \ker (2\lambda)$ is slightly bigger than $K_0$. Both $\Aut^{\circ}(X/B)$ and $K_0 = 2 \cdot K_0[2]$ contained in $K_0[2]$ are invariant under deformations, so the inclusion $K_0 \subset \Aut^{\circ}(X/B)$ in \Cref{prop:polarization scheme inclusion Kum} is preserved under deformation. The claim follows.
	\end{proof}
	
	\begin{remark} \label{rmk:polarization scheme equality}
		We may state \Cref{thm:polarization scheme equality} in the following simpler way: we have an equality of group schemes
		\[ K_0 = \Aut' (X/B) \quad \big( := \Aut^{\circ}(X/B) \cap \Aut'(X) \big) ,\]
		where $\Aut'(X) \subset \Aut^{\circ}(X)$ is a group defined in \Cref{rmk:aut0 new definition}. For most of the known examples of Lagrangian fibered hyper-K\"ahler manifolds, we have $\Aut' (X/B) = \Aut^{\circ} (X/B)$. There is a single known example where the inclusion $\Aut'(X/B) \subset \Aut^{\circ}(X/B)$ is strict, when $X$ is of $\Kum_3$-type and $\pi$ has the polarization type $(1,2,2)$. In this case, $\Aut'(X/B) \cong (\ZZ/2)^{\oplus 4}$ and $\Aut^{\circ}(X/B) \cong (\ZZ/2)^{\oplus 5}$.
	\end{remark}
	
	A direct consequence of this theorem is a promised compactification of the dual torus fibration $\check \pi : \check X_0 \to B_0$.
	
	\begin{theorem} \label{thm:dual Lagrangian fibration}
		Let $\pi : X \to B$ be a Lagrangian fibration of a compact hyper-K\"ahler manifold of $\text{K3}^{[n]}$, $\Kum_n$, OG10 or OG6-type. Then
		\[ \check \pi : \check X \to B \qquad \mbox{for} \quad \check X = X / K \]
		defines a compactification of the dual torus fibration $\check \pi : \check X_0 \to B_0$.
	\end{theorem}
	\begin{proof}
		As explained in the introduction, we have defined the dual torus fibration by $\check X_0 = X_0 / K_0$. For known deformation types, \Cref{thm:polarization scheme equality} proved that $K_0$ extends to a constant group scheme $K$ over $B$ acting on $X$. Therefore, the group scheme quotient $X_0 / K_0 \to B_0$ can be compactified into $X / K \to B$. Since $K \to B$ is a constant group scheme, the quotient $X / K$ may be considered either as a group scheme quotient over $B$ or a finite group quotient over $\CC$.
	\end{proof}
	
	When $X$ is of $\text{K3}^{[n]}$ or OG10-type, $\check X$ is identical to $X$ and there is nothing more to say. Let us study more on the space $\check X$ when $X$ is of $\Kum_n$ or OG6-type. Being a quotient by $H^2$-trivial automorphisms, $\check X$ inherits many interesting properties from $X$. We provide an \cref{sec:quotient} to collect their properties in a more general setup; the following proposition is a direct consequence of this more general discussion. For definitions of a primitive symplectic orbifold and irreducible symplectic variety used in the following proposition, see \Cref{sec:singular HK}.
	
	\begin{proposition} \label{prop:dual HK}
		Keep the notation from \Cref{thm:dual Lagrangian fibration}, and assume $X$ is either of $\Kum_n$ or OG6-type. Then
		\begin{enumerate}
			\item $\check X$ is a compact primitive symplectic orbifold and also an irreducible symplectic variety.
			\item $\check X$ does not admit a symplectic resolution.
			\item $\check X$ is simply connected. It has the Fujiki constant $c_{\check X} = 1/c_X$.
			\item $H^2 (\check X, \QQ)$ and $H^2 (X, \QQ)$ are Hodge isomorphic and Beauville--Bogomolov isometric.
			\item The LLV algebras and Mumford--Tate algebras of $X$ and $\check X$ are isomorphic.
			\item The pullback $H^* (\check X, \QQ) \to H^* (X, \QQ)$ is an injective map of LLV structures.
			\item If $\mathcal X \to \Def(X)$ is the universal deformation of $X$, then $\mathcal X / K \to \Def(X)$ is the (locally trivial) universal deformation of $\check X$.
		\end{enumerate}
	\end{proposition}
	\begin{proof}
		Everything is a direct consequence of \Cref{prop:appendix 1} and \ref{prop:appendix 2}. Only the first three items need further explanations. For the first and second items, it is enough to show $\codim X^f \ge 4$ for all $f \in K \setminus \{ \id \}$. We will see later in \Cref{prop:fixed locus is deformation equivalent} that the fixed loci of $H^2$-trivial automorphisms deform when $X$ deform. Hence we may prove this for any model in the deformation class on $X$. For OG6, the fixed loci are computed in \cite[\S 6]{mon-wan17}; they are either K3 surfaces or points. For $\Kum_n$, the fixed loci are computed in \cite[Lem 3.5]{ogu20}, and similarly one can deduce their codimension is always $\ge 4$. For the third item, simply notice the group $K$ has order $c_X^2$ in all cases.
	\end{proof}
	
	The proposition shows $\check X$ has quotient singularities when $X$ is of $\Kum_n$ or OG6-type. Therefore, $\check X$ cannot be homeomorphic to $X$. We call the corresponding $\check X$ in each case the \emph{dual Kummer variety} and \emph{dual OG6}, respectively.
	
	Finally, the proposition shows in particular the local deformation behavior and period domains of $X$ and $\check X$ are identical. Therefore, one can still apply the method in \cite[\S 2]{gross-tos-zhang13} at the level of period domains and obtain similar conclusions for all known deformation types of hyper-K\"ahler manifolds.
	One subtlety here is that the quotient construction works for any deformation $X'$ of $X$, even if $X'$ does not admit any Lagrangian fibration; the quotient $X'/K$ is still well-defined because we have considered $K$ as an abstract subgroup of the group $\Aut^{\circ}(X)$. The local universal deformation space of the Lagrangian fibration $\pi : X \to B$ is a hyperplane $\Def(X, H) \subset \Def(X)$ (see \cite{mat16}). Once we choose a deformation $X'$ by respecting the Lagrangian fibration $[\pi' : X' \to B'] \in \Def(X, H)$, we can say $\check \pi' : X'/K \to B'$ is the dual Lagrangian fibration of $\pi' : X' \to B'$.

	\section{Example: the dual Kummer fourfolds} \label{sec:example dual Kum2}
	To illustrate the geometry of dual Lagrangian fibrations more concretely, we focus on the simplest nontrivial case of \Cref{thm:dual Lagrangian fibration}: when $X$ is of $\Kum_2$-type. Throughout, we let $X$ to be a $\Kum_2$-type hyper-K\"ahler fourfold and $\pi : X \to B = \PP^2$ its Lagrangian fibration. We will use all the results in previous sections without mentioning them explicitly. We write for simplicity
	\[ K = \Aut^{\circ} (X/B) .\]
	There exist isomorphisms $\Aut^{\circ} (X) \cong \ZZ/2 \ltimes (\ZZ/3)^{\oplus 4}$ and $K \cong (\ZZ/3)^{\oplus 2}$. Again for simplicity, we call $f \in \Aut^{\circ}(X)$ a \emph{translation} if $f$ does not contain a $\ZZ/2$-part, and call $f$ an \emph{involution} if $f$ has a nontrivial $\ZZ/2$-part. If $f \neq \id$ is a translation (resp., involution) then it has order $3$ (resp., order $2$). There are precisely $81$ translations and $81$ involutions. The subgroup $K \subset \Aut^{\circ}(X)$ consists of $9$ translations respecting the Lagrangian fibration. We define the dual Kummer fourfold by $\check \pi : \check X = X/K \to B$. The main result of this section will be \Cref{prop:dual Kum2}. It collects some more precise geometric and cohomological descriptions of $\check X$. Similar method may apply to the OG6-type and higher dimensional $\Kum_n$-types.
	
	We will use the notion of the LLV structure to describe the cohomology of $\check X$. To do so, we first need to review the LLV structure of the generalized Kummer fourfolds (following \cite{gklr22}). Recall that the Beauville--Bogomolov quadratic space of $X$ (and hence $\check X$) is isomorphic to $(H^2 (X, \QQ), q) \cong U^{\oplus 3} \oplus \langle -6 \rangle$ where $U$ denotes the hyperbolic plane $\begin{psmallmatrix} 0 & 1 \\ 1 & 0 \end{psmallmatrix}$. For simplicity, we denote $\bar V = H^2 (X, \QQ)$ and its Mukai completion by
	\[ (V, \tilde q) = (\bar V, q) \oplus U \quad \big( \cong U^{\oplus 4} \oplus \langle -6 \rangle \big) .\]
	Set $\mathfrak g \cong \so (V, \tilde q)$ and $\bar {\mathfrak g} \cong \so (\bar V, q)$ the LLV algebra and reduced LLV algebra of $X$ (and $\check X$). It is a \emph{split} semisimple $\QQ$-Lie algebra. Associated to any dominant weight $\mu$ of $\mathfrak g$, there exists an irreducible $\mathfrak g$-module $V_{\mu}$ over $\QQ$. The LLV structure of $X$ is explicitly computed in \cite[(4.7)]{loo-lunts97}; we have an isomorphism of $\mathfrak g$-modules
	\[ H^* (X, \QQ) \cong V_{(2)} \oplus 80\QQ \oplus V_{(\frac{1}{2},\frac{1}{2},\frac{1}{2},\frac{1}{2})} .\]
	It is sometimes convenient to consider the reduced LLV structure on the fixed degree cohomologies $H^k (X, \QQ)$. The LLV structure restricts to the reduced LLV structure on the middle cohomology: there exists an isomorphism of $\bar {\mathfrak g}$-modules
	\begin{equation} \label{eq:reduced LLV structure of Kum2}
		H^4 (X, \QQ) \cong \bar V_{(2)} \oplus 81\QQ .
	\end{equation}
	
	Let us now state the main result of this section.
	
	\begin{proposition} \label{prop:dual Kum2}
		Let $\pi : X \to B$ be a Lagrangian fibration of a $\Kum_2$-type hyper-K\"ahler fourfold and $\check \pi : \check X \to B$ its dual fibration. Then (in addition to \Cref{prop:dual HK})
		\begin{enumerate}
			\item $\check X$ has precisely $36$ isolated cyclic quotient singularities of type $\frac{1}{3} (1,1,2,2)$.
			\item $\check X$ (and any of its deformation) contains $9$ smooth K3 surfaces. Each of them passes through $4$ singularities of $\check X$. The image of each K3 surface by $\check \pi$ is a line in $B = \PP^2$.
			\item The LLV decomposition of the cohomology of $\check X$ is
			\[ H^* (\check X, \QQ) \cong V_{(2)} \oplus 8\QQ \oplus V_{(\frac{1}{2},\frac{1}{2},\frac{1}{2},\frac{1}{2})} .\]
		\end{enumerate}
	\end{proposition}
	
	The rest of the section will be devoted to the proof of \Cref{prop:dual Kum2}. The following proposition claims any Lagrangian fibered $\Kum_2$-type hyper-K\"ahler manifolds are deformation equivalent. The idea originates from the results of Markman (e.g., \cite[Prop 1.7]{mar14}).
	
	\begin{lemma} \label{lem:Kum2 Lagrangian fibrations are deformation equivalent}
		Any Lagrangian fibration of a $\Kum_2$-type hyper-K\"ahler fourfold $\pi : X \to B$ is deformation equivalent to each other.
	\end{lemma}
	\begin{proof}
		The polarization type of $\pi$ is $(1,3)$ by \Cref{thm:polarization type computation}. Setting $h \in H^2 (X, \ZZ)$ to be an associated cohomology class of $\pi^* \mathcal O_B(1)$, its divisibility $\operatorname{div}(h)$ is $1$ by \cite[Thm 1.1]{wie18}. We can now imitate the method of \cite[\S 7]{mon-rap21}. The lattice theory result in \cite[Lem 2.6]{mon-rap21} forces any two primitive isotropic elements $h, h'$ with divisibility $1$ in $H^2 (X, \ZZ)$ are monodromy equivalent. We can imitate the proof of \cite[Thm 7.2]{mon-rap21} (or the proof of \cite[Prop 1.7]{mar14}) and show that any pairs $(X, H)$ with a primitive isotropic $H$ with divisibility $1$ are deformation equivalent. This proves the claim.
	\end{proof}
	
	Thanks to this proposition, we can often specialize our discussion to a single model. Our explicit model for Lagrangian fibered $\Kum_2$-type hyper-K\"ahler manifolds is the following example presented in \cite[\S 2]{mat15}. Let $E$ and $E'$ be elliptic curves and $S = E' \times E$ be an abelian surface. Consider a commutative diagram
	\begin{equation} \label{diag:Kum2}
	\begin{tikzcd}[sep=tiny]
		S^{[3]} \arrow[rd, "{(\Sigma \circ \pr_1, \, \pr_2)}"] \arrow[dd, "\Alb"'] \\
		& E' \times E^{(3)} \arrow[ld, "\id \times \Sigma"] \\
		E' \times E
	\end{tikzcd},
	\end{equation}
	where $\pr_1 : S^{[3]} \to (E')^{(3)}$ and $\pr_2 : S^{[3]} \to E^{(3)}$ are the coordinate projections and $\Sigma$ are the summation maps. By the discussion we had in \Cref{sec:aut0 computation}, this is an isotrivial family of Lagrangian fibered hyper-K\"ahler manifolds of $\Kum_2$-type. The advantage of this construction to the moduli construction in \Cref{sec:aut0 computation} is that this gives us an honest generalized Kummer fourfold, and thus we can use the computational results in \cite{has-tsc13} and \cite{ogu20}.
	
	Let us also recall some known facts about the fixed loci of $H^2$-trivial automorphisms.
	
	\begin{lemma} \label{prop:fixed locus is deformation equivalent}
		Let $X$ be a compact hyper-K\"ahler manifold and $G \subset \Aut^{\circ} (X)$ any subgroup. If $X'$ is deformation equivalent to $X$ then $(X')^G$ is deformation equivalent to $X^G$.
	\end{lemma}
	\begin{proof}
		Let $p : \mathcal X \to \Def(X)$ be a universal deformation of $X$. Since $G$ acts fiberwise on $p$, the morphism $\mathcal X^G \to \Def(X)$ gives a family of fixed loci $(X_t)^G$. Because $G$ is a finite group acting on a complex manifold $\mathcal X$, its fixed locus $\mathcal X^G$ is a complex manifold. Similarly, each $(X_t)^G$ is a complex (symplectic) manifold. Hence $\mathcal X^G \to \Def(X)$ is a smooth proper family and the claim follows.
	\end{proof}
	
	\begin{lemma}
		Let $X$ be a $\Kum_2$-type hyper-K\"ahler manifold and $f \in \Aut^{\circ}(X)$ its $H^2$-trivial automorphism.
		\begin{enumerate}
			\item If $f$ is an involution then its fixed locus $X^f$ is a disjoint union of a K3 surface and $36$ points.
			\item If $f \neq \id$ is a translation then its fixed locus $X^f$ consists of $27$ points.
		\end{enumerate}
	\end{lemma}
	\begin{proof}
		When $X$ is an honest generalized Kummer fourfold the statements were proved in \cite[Lem 3.5]{ogu20}, \cite[Thm 4.4]{has-tsc13} and \cite[Thm 7.5]{kap-men18}. If we deform $X$ then the fixed locus $X^f$ also deforms by \Cref{prop:fixed locus is deformation equivalent}.
	\end{proof}
	
	Following \cite[Thm 4.4]{has-tsc13}, any $\Kum_2$-type hyper-K\"ahler manifold must always contain $81$ K3 surfaces obtained by the fixed loci of $81$ involutions (this was first observed in \cite[\S 6]{kal-ver98}). The $81$ K3 surfaces are related by $81$ translations, and represent the $81$ trivial reduced LLV classes in \eqref{eq:reduced LLV structure of Kum2}. With these backgrounds, we can begin the proof of \Cref{prop:dual Kum2}.
	
	\begin{proof} [Proof of \Cref{prop:dual Kum2}]
		\Cref{prop:appendix 1} says the singularity locus of $\check X$ is the image of the set $S = \bigcup_{f \in G \setminus \{ \id \}} X^f$. The set $X^f$ consists of $27$ points for $f \neq \id$ by the lemma above. This means $S$ consists of $27 \times 4 = 108$ points. The quotient map $p : X \to \check X$ identifies $3$ points to a single point of $\ZZ/3$-quotient singularity. This proves there are $108 / 3 = 36$ $\ZZ/3$-quotient singularities.\footnote{Our original computation was incorrect. This was pointed out in \cite[Ex 3.6]{bec-song22}.}
		
		Any symplectic $\ZZ/3$-quotient singularity must be of type $\frac{1}{3}(1,1,2,2)$. By \cite[Prop 6]{prill67}, the $\ZZ/3$-action is locally biholomorphic to a symplectic linear action on $\CC^4$ around $0$. Its eigenvalues are either $1$, $\zeta$ and $\zeta^2$ where $\zeta$ is the third primitive root of unity. In this case, $1$ cannot arise because the fixed locus of the action should be the origin. Hence there are five possibilities of the linear action up to conjugate, and one easily checks $\operatorname{diag}(\zeta, \zeta, \zeta^2, \zeta^2)$ is the only symplectic linear map among them (for some symplectic form).
		
		For the second item, notice first that the $81$ K3 surfaces in $X$ are identified into $9$ K3 surfaces in $\check X$. Let us deform the Lagrangian fibration and assume we are in the construction \eqref{diag:Kum2} (\Cref{lem:Kum2 Lagrangian fibrations are deformation equivalent}). The $81$ K3 surfaces are explicitly described in this case by \cite[Thm 4.4]{has-tsc13}. One explicitly computes each of $81$ K3 surface passes through four points in $\bigcup_{f \in G \setminus \{ \id \}} X^f$, and the four points are not identified by the quotient map $p$. Hence each of nine K3 surfaces in $\check X$ passes through four singluar points of $\check X$. (Note: the nine K3 surfaces in $\check X$ do intersect each others, but the intersections are smooth points in $\check X$.) One finally checks the image of each K3 surface under $\pi$ can be considered as a sublinear system in $\PP^2$, so it is a line.
		
		For the last item, recall $H^* (\check X, \QQ) = H^* (X, \QQ)^K$. The translations in $K$ act trivially on $H^2 (X, \QQ)$ by definition and trivially on $H^3 (X, \QQ)$ by the computations in \cite[\S 3]{ogu20}. Hence we only need to prove $H^4 (X, \QQ)^K \cong \bar V_{(2)} \oplus 9\QQ$. The Verbitsky component is preserved by $K$, so $\bar V_{(2)}$ is $K$-invariant. Again recall the $81$ trivial reduced LLV classes in $H^4 (X, \QQ)$ were represented by $81$ K3 surfaces which are bound by $81$ translation automorphisms. Since only $9$ of them survives in $\check X$, the fourth cohomology is as desired.
	\end{proof}

	\appendix

	\section{Various notions of singular hyper-K\"ahler varieties} \label{sec:singular HK}
	Many of the important properties of compact hyper-K\"ahler manifolds have been generalized to singular settings. There are several definitions of singular hyper-K\"ahler varieties in the current literature. To make our discussion less ambiguous, we collect some definitions and compare them. Our main references are \cite{bak-lehn22}, \cite{sch20} and \cite{menet20}.
	
	If $X$ is a normal complex space, then its sheaf of reflexive $k$-forms is defined to be the reflexive closure of the sheaf of $k$-forms $\Omega_X^{[k]} = (\Omega_X^k)^{\vee\vee}$, or equivalently $\Omega_X^{[k]} = j_* \Omega_{X_{\operatorname{reg}}}^k$ where $j : X_{\operatorname{reg}} \hookrightarrow X$ is the smooth locus of $X$. A quasi-\'etale morphism is a morphism \'etale outside of a codimension $\ge 2$ closed subvariety.
	
	\begin{definition} [{\cite[Def 3.1]{bak-lehn22}, \cite[Def 1]{sch20}, \cite[Def 3.1]{menet20}}] \label{def:singular HK}
		Let $X$ be a compact normal K\"ahler space and $\sigma \in H^0 (X, \Omega_X^{[2]})$ a reflexive $2$-form.
		\begin{enumerate}
			\item $(X, \sigma)$ is called a \emph{symplectic variety} if $X$ has rational singularities and $\sigma$ is nondegenerate on $X_{\operatorname{reg}}$.
			
			\item $X$ is called a \emph{primitive symplectic variety} if
			\[ H^0 (X, \Omega_X^{[1]}) = 0, \qquad H^0 (X, \Omega_X^{[2]}) = \CC \sigma, \]
			and $(X, \sigma)$ is a symplectic variety.
			
			\item $X$ is called an \emph{irreducible symplectic variety} if it is a primitive symplectic variety with the following condition: for any finite quasi-\'etale cover $f : X' \to X$, we have
			\[ H^0 (X', \Omega_{X'}^{[2k+1]}) = 0, \qquad H^0 (X', \Omega_{X'}^{[2k]}) = \CC \cdot f^* \sigma^{[k]} \qquad \mbox{for} \quad k \ge 0 .\]
			
			\item $X$ is called a \emph{Namikawa symplectic variety} if it is a $\QQ$-factorial and terminal primitive symplectic variety.
			
			\item $X$ is called a \emph{primitive symplectic orbifold} if it is Namikawa symplectic with only finite quotient singularities.
		\end{enumerate}
	\end{definition}
	
	We have a series of implications
	\[\begin{tikzcd}
		\mbox{primitive symplecitc orbifold} \arrow[r, Rightarrow] & \mbox{Namikawa symplectic} \arrow[d, Rightarrow] \\
		\mbox{irreducible symplectic} \arrow[r, Rightarrow] & \mbox{primitive symplectic} \arrow[r, Rightarrow] & \mbox{symplectic.}
	\end{tikzcd}\]
	Eventually, the dual hyper-K\"ahler variety $\check X$ in \Cref{thm:introduction} will be both a primitive symplectic orbifold and an irreducible symplectic variety (\Cref{prop:dual HK}). Hence all of the discussions here apply.
	
	Many of the interesting properties of compact hyper-K\"ahler manifolds generalize to their singular analogues, especially to primitive symplectic varieties. We highlight some of their properties that will be useful to our discussion. Let $X$ be a primitive symplectic variety of dimension $2n$.
	\begin{itemize}
		\item The normailzation of the singular locus $X_{\operatorname{sing}}$ is again symplectic \cite{kal06}. In particular, $X_{\operatorname{sing}}$ is always even dimensional.
		\item There exist a notion of the Beauville--Bogomolov form and Fujiki constant of $X$, so that the Fujiki relation \eqref{eq:fujiki relation} holds \cite[Thm 2]{sch20} \cite[Prop 5.20]{bak-lehn22}.
		\item $X$ is Namikawa symplectic if and only if it is $\QQ$-factorial and $\codim X_{\operatorname{sing}} \ge 4$ \cite{nam01} \cite[Thm 3.4]{bak-lehn22}.
		\item Every morphism $\pi : X \to B$ with connected fibers to a normal base $B$ (with $0 < \dim B < 2n)$ is a Lagrangian fibration \cite[Thm 3]{sch20}. That is, all the irreducible components of the fibers of $\pi$ are Lagrangian subvarieties of $X$.
		\item The Hodge structure $H^2 (X, \ZZ)$ is pure \cite[Thm 8]{sch20} \cite[Cor 3.5]{bak-lehn22}. If $X$ is a primitive symplectic orbifold, then the full cohomology $H^* (X, \QQ)$ is a pure Hodge structure.
		\item There exists a universal locally trivial deformation $\mathcal X \to \Def^{lt}(X)$ over a smooth complex germ $\Def^{lt}(X)$ of dimension $h^{1,1}(X)$ \cite[Thm 4.7]{bak-lehn22}. If $X$ is Namikawa symplectic, then any deformation is automatically locally trivial \cite{nam06}.
		\item The local Torelli theorem holds for $\Def^{lt}(X)$. In fact, global Torelli theorem holds in a suitable form \cite{bak-lehn22}.
	\end{itemize}
	We will use these facts in \Cref{sec:dual Lagrangian fibration} and \Cref{sec:quotient}, without mentioning them explicitly.

	\section{Quotient of a hyper-K\"ahler manifold by $H^2$-trivial automorphisms} \label{sec:quotient}
	Let $X$ be a compact hyper-K\"ahler manifold and $\Aut^{\circ} (X)$ the finite group of $H^2$-trivial automorphisms. Throughout the appendix, we always let
	\[ G \subset \Aut^{\circ} (X) \]
	to be \emph{any} subgroup and write
	\begin{eqnarray} \label{eq:quotient}
		p : X \to \bar X = X/G .
	\end{eqnarray}
	The goal of this appendix is to gather basic geometric and cohomological properties of the quotient $\bar X$. Note that Lagrangian fibrations play no role in this appendix. The main results are \Cref{prop:appendix 1} and \ref{prop:appendix 2}.
	
	\begin{proposition} \label{prop:appendix 1}
		Consider the quotient \eqref{eq:quotient} of a compact hyper-K\"ahler manifold $X$.
		\begin{enumerate}
			\item The morphism $p$ is a finite quasi-\'etale symplectic quotient.
			\item $\bar X$ is a $\QQ$-factorial irreducible symplectic variety whose singularity locus is $p \Big( \bigcup_{f \in G \setminus \{ \id \}} X^f \Big)$. If $\codim X^f > 2$ for all $f \in G \setminus \{ \id \}$, then $\bar X$ is also a primitive symplectic orbifold.
			\item $\bar X$ is simply connected.
			\item If $\mathcal X \to \Def(X)$ is the universal deformation of $X$ then the quotient $\mathcal X/G \to \Def(X)$ becomes the universal locally trivial deformation of $\bar X$.
		\end{enumerate}
	\end{proposition}
	
	The quotient $\bar X = X / G$ being an irreducible symplectic variety, its behavior is intimately related to its (second) cohomology. To talk about the precise cohomological behavior of $\bar X$, we first need to fix the Beauville--Bogomolov form; the Beauville--Bogomolov form is a priori only defined up to scalar. We define a symmetric bilinear form $q_{\bar X} : H^2 (\bar X, \ZZ) \otimes H^2 (\bar X, \ZZ) \to \ZZ$ by
	\[ q_{\bar X} (x, y) = q_X (p^* x, p^* y) \qquad\mbox{for}\quad x, y \in H^2 (\bar X, \ZZ) .\]
	The reader should be aware that $q_{\bar X}$ may be a \emph{non-primitive} bilinear form with this definition.
	
	\begin{proposition} \label{prop:appendix 2}
		Notation as above.
		\begin{enumerate}
			\item $q_{\bar X}$ is a Beauville--Bogomolov form of $\bar X$. The Fujiki constant of $\bar X$ is $c_{\bar X} = c_X/|G|$.
			\item The pullback
			\[ p^* : H^2(\bar X, \ZZ)/(\mbox{torsion}) \to H^2 (X, \ZZ) \]
			is an injective Hodge structure homomorphism and a Beauville--Bogomolov isometry. It is an isomorphism over $\QQ$.
			\item The LLV algebra of $X$ and $\bar X$ are canonically isomorphic. Denoting them by $\mathfrak g$, the pullback
			\[ p^* : H^* (\bar X, \QQ) \to H^* (X, \QQ) \]
			is an injective $\mathfrak g$-module homomorphism.
			\item For all $k$, the special Mumford--Tate algebra of $H^k (\bar X, \QQ)$ is isomorphic to that of $H^2 (X, \QQ)$. As a consequence, any $\mathfrak g$-module decomposition of $H^* (\bar X, \QQ)$ is a pure Hodge structure decomposition.
		\end{enumerate}
	\end{proposition}
	
	Note again that the subgroup $G \subset \Aut^{\circ} (X)$ was taken arbitrary. Hence we have a family of irreducible symplectic varieties corresponding to each subgroups of $\Aut^{\circ} (X)$. That is, we get a Galois correspondence between the subgroups $G \subset \Aut^{\circ} (X)$ and the symplectic quotients $\bar X = X/G$ with the same rational Beauville--Bogomolov forms. In particular, their deformation behaviors are all identical.
	
	The rest of this appendix is devoted to the proof of \Cref{prop:appendix 1} and \ref{prop:appendix 2}. Most of the proofs will be straightforward so we will be brief.
	
	\begin{proof} [Proof of \Cref{prop:appendix 1}: Part 1]
		Let us present the proof of the theorem without the second item. The second item will be proved separately in Part 2.
		
		The group $G$ acts trivially on $H^2 (X, \ZZ)$, so it acts symplectically on $X$. Hence $p$ is a symplectic quotient. The ramified locus of $p$ is contained in the union of the fixed loci $\bigcup_{f \in G \setminus \{ \id \}} X^f$, which is of codimension $\ge 2$. This means $p$ is quasi-\'etale and the first item follows.
		
		The third item is a direct consequence of the second item, because any irreducible symplectic variety is simply connected by \cite[Cor 13.3]{greb-gue-keb19}. The last item again follows directly from \cite[Thm 3.5, Lem 3.10]{fuj83}. Since $G$ acts on $\mathcal X$ holomorphically and trivially on $H^2 (X, \ZZ)$, $\mathcal X \to \Def(X)$ equipped with a $G$-action is the universal deformation of the pair $(X, G)$. Once we have a universal deformation of the pair $(X, G)$, the quotient $\mathcal X/G \to \Def(X)$ is the locally trivial universal family of $X/G$.
	\end{proof}
	
	\begin{lemma} \label{lem:quasi etale is primitive symplectic}
		Let $(X, \sigma)$ be a compact symplectic variety and $f : X' \to X$ a finite quasi-\'etale morphism. Then $(X', f^* \sigma)$ is a compact symplectic variety.
	\end{lemma}
	\begin{proof}
		By \cite[Prop 5.20]{kol-mori} or \cite[Rmk 3.4]{greb-keb-pet16}, $X'$ is Gorenstein and canonical. Therefore, it has rational singularities by \cite[Cor 5.24]{kol-mori}. Now $f^* \sigma \in H^0 (X', \Omega_{X'}^{[2]})$ is a symplectic form in codimension $1$ as $f$ is \'etale in codimension $1$. The claim follows.
	\end{proof}
	
	\begin{proof} [Proof of \Cref{prop:appendix 1}: Part 2]
		We prove the second item here. As a finite quotient of a smooth variety $X$, the space $\bar X$ is certainly $\QQ$-factorial and has quotient singularities. Fix a point $x \in X$ and let $\bar x = p(x)$. According to the Chevalley–Shephard–Todd theorem, the quotient $\bar X$ is smooth at $\bar x$ if and only if the stabilizer group $G_x$ acting on the tangent space $T_x X$ is generated by pseudoreflections (i.e., linear automorphisms on $T_x X$ with codimension $1$ fixed loci). If $x \in X$ has a nontrivial stabilizer $G_x$, any nontrivial automorphism $f \in G_x$ is symplectic so has codimension $\ge 2$ fixed locus. This means $G_x$ cannot be generated by pseudoreflections. Therefore, $\bar X$ is singular at $\bar x$. If we further assume $\codim X^f \ge 4$ for all nontrivial $f \in G$, then $\codim \bar X_{\operatorname{sing}} \ge 4$ and $\bar X$ becomes Namikawa symplectic.
		
		To prove $\bar X$ is irreducible symplectic, we follow the argument of Matsushita \cite[Lem 2.2]{mat15}. Let $f : \bar Y \to \bar X$ be an arbitrary finite quasi-\'etale morphism. Consider the diagram
		\[\begin{tikzcd}
			Y \arrow[r, "g"] \arrow[d, "q"] & X \arrow[d, "p"] \\
			\bar Y \arrow[r, "f"] & \bar X
		\end{tikzcd},\]
		where $Y$ is the normalization of the fiber product $X \times_{\bar X} \bar Y$. We claim $g$ and $q$ are finite quasi-\'etale. The finiteness is clear, so we concentrate on their quasi-\'etaleness. Notice that the quasi-\'etale property is stable under base change, so we need to prove the normalization in this case is quasi-\'etale. But notice that $X$ is smooth and $f$ is quasi-\'etale, so that $X \times_{\bar X} \bar Y$ is smooth in codimension $1$. Hence the normalization of it is in fact isomorphism in codimension $1$. This proves $g$ and $q$ are quasi-\'etale.
		
		Now $X$ is smooth, $Y$ is normal, and $g : Y \to X$ is finite quasi-\'etale. By the Zariski--Nagata purity theorem of the branch locus (e.g., \cite[Tag 0BMB]{stacks-project}), this forces $g$ to be \'etale. The hyper-K\"ahler manifold $X$ is simply connected, so this means $Y$ must be a disjoint union of several isomorphic copies of $X$. Let us fix a connected component $Y_0$ of $Y$. It is a hyper-K\"ahler manifold isomorphic to $X$.
		
		Consider the morphism $q$ restricted to the connected component $q : Y_0 \to \bar Y$. It is a finite quasi-\'etale morphism. Note that the target $\bar Y$ is canonical (\Cref{lem:quasi etale is primitive symplectic}), so \cite[Thm 4.3]{gkkp11} guarantees the existence of a reflexive pullback $q^* : H^0 (\bar Y, \Omega_{\bar Y}^{[k]}) \to H^0 (Y_0, \Omega_{Y_0}^{[k]})$. Since $q$ is quasi-\'etale, this morphism is injective. But recall that $Y_0 \cong X$ is a hyper-K\"ahler manifold, so this forces $\bar Y$ to satisfy the dimension condition of the definition of irreducible symplectic varieties. This proves $\bar X$ is an irreducible symplectic variety.
	\end{proof}
	
	\begin{proof} [Proof of \Cref{prop:appendix 2}]
		The following sequence of identities proves $q_{\bar X}$ is the Beauville--Bogomolov form with the Fujiki constant $c_{\bar X} = c_X/|G|$:
		\[ \int_{\bar X} x^{2n} = \frac{1}{|G|} \int_X (p^*x)^{2n} = \frac{c_X}{|G|} q_X (p^*x)^n = \frac{c_X}{|G|} q_{\bar X} (x)^n .\]
		Since $\bar X$ is a compact K\"ahler orbifold, its rational singular cohomology admits a well-behaved pure Hodge structure (e.g., \cite[\S 2.5]{pet-ste}) and $p^* : H^* (\bar X, \QQ) \to H^* (X, \QQ)$ is an injective Hodge structure homomorphism with the image $H^* (X, \QQ)^G$. In particular, $p^*$ is an isomorphism in degree $2$.
		
		To prove $H^* (\bar X, \QQ) = H^* (X, \QQ)^G$ is closed under the $\mathfrak g$-action, it is enough to prove the $G$-action and $\mathfrak g$-action on $H^* (X, \QQ)$ commutes. Recall that the LLV structure is diffeomorphism invariant. In other words, if $f : X_1 \to X_2$ is a diffeomorphism between two compact hyper-K\"ahler manifolds then we have
		\[ f^* (L_x (\xi)) = L_{f^* x} (f^* \xi), \qquad f^* (\Lambda_x (\xi)) = \Lambda_{f^* x} (f^* \xi) \]
		for any $x \in H^2 (X_2, \QQ)$ and $\xi \in H^* (X_2, \QQ)$. Here $L_x$ and $\Lambda_x$ are Lefschetz and inverse Lefschetz operators associated to $x$. If we set $X_1 = X_2 = X$ and $f \in G$ to be an $H^2$-trivial automorphism then this means $f^*$ commutes with the operators $L_x$ and $\Lambda_x$. That is, $G$ commutes with $\mathfrak g$.
		
		To obtain the results about the Mumford--Tate algebras one imitates the method used in \cite[\S 2]{gklr22} and deduces $f \in \mathfrak g$ for $f$ a Weil operator on the cohomology $H^* (\bar X, \QQ)$ (which is the restriction of Weil operator on $H^* (X, \QQ)$). This proves all the special Mumford--Tate algebra of $H^k (\bar X, \QQ)$ are the same and even same for that of $H^2 (X, \QQ)$.
	\end{proof}

	\bibliographystyle{amsalpha}
	\bibliography{DualHK.bib}

\providecommand{\bysame}{\leavevmode\hbox to3em{\hrulefill}\thinspace}
\providecommand{\MR}{\relax\ifhmode\unskip\space\fi MR }
\providecommand{\MRhref}[2]{%
  \href{http://www.ams.org/mathscinet-getitem?mr=#1}{#2}
}
\providecommand{\href}[2]{#2}
\begin{thebibliography}{BNWS11}

\bibitem[AF16]{ari-fed16}
D.~Arinkin and R.~Fedorov, \emph{Partial {F}ourier-{M}ukai transform for
  integrable systems with applications to {H}itchin fibration}, Duke Math. J.
  \textbf{165} (2016), no.~15, 2991--3042.

\bibitem[B\"20]{bul20}
T.-H. B\"{u}lles, \emph{Motives of moduli spaces on {K}3 surfaces and of
  special cubic fourfolds}, Manuscripta Math. \textbf{161} (2020), no.~1-2,
  109--124.

\bibitem[Bar87]{barth87}
W.~Barth, \emph{Abelian surfaces with {$(1,2)$}-polarization}, Algebraic
  geometry, {S}endai, 1985, Adv. Stud. Pure Math., vol.~10, North-Holland,
  Amsterdam, 1987, pp.~41--84.

\bibitem[Bea83a]{bea83b}
A.~Beauville, \emph{Some remarks on {K}\"{a}hler manifolds with {$c_{1}=0$}},
  Classification of algebraic and analytic manifolds ({K}atata, 1982), Progr.
  Math., vol.~39, Birkh\"{a}user Boston, Boston, MA, 1983, pp.~1--26.

\bibitem[Bea83b]{bea83}
\bysame, \emph{Vari\'{e}t\'{e}s {K}\"{a}hleriennes dont la premi\`ere classe de
  {C}hern est nulle}, J. Differential Geom. \textbf{18} (1983), no.~4, 755--782
  (1984).

\bibitem[BL04]{bir-lan:abelian}
C.~Birkenhake and H.~Lange, \emph{Complex abelian varieties}, second ed.,
  Grundlehren der Mathematischen Wissenschaften, vol. 302, Springer-Verlag,
  Berlin, 2004.

\bibitem[BL22]{bak-lehn22}
B.~Bakker and C.~Lehn, \emph{The global moduli theory of symplectic varieties},
  J. Reine Angew. Math. (2022).

\bibitem[BLR90]{neron}
S.~Bosch, W.~L\"{u}tkebohmert, and M.~Raynaud, \emph{N\'{e}ron models},
  Ergebnisse der Mathematik und ihrer Grenzgebiete (3), vol.~21,
  Springer-Verlag, Berlin, 1990.

\bibitem[BNWS11]{boi-nei-sar11}
S.~Boissi\`ere, M.~Nieper-Wi{\ss}kirchen, and A.~Sarti, \emph{Higher
  dimensional {E}nriques varieties and automorphisms of generalized {K}ummer
  varieties}, J. Math. Pures Appl. (9) \textbf{95} (2011), no.~5, 553--563.

\bibitem[Bri11]{brion11}
M.~Brion, \emph{On automorphism groups of fiber bundles}, Publ. Mat. Urug.
  \textbf{12} (2011), 39--66.

\bibitem[Bri18]{brion18}
\bysame, \emph{Linearization of algebraic group actions}, Handbook of group
  actions. {V}ol. {IV}, Adv. Lect. Math. (ALM), vol.~41, Int. Press,
  Somerville, MA, 2018, pp.~291--340.

\bibitem[BS22]{bec-song22}
T.~Beckmann and J.~Song, \emph{Second {C}hern class and {F}ujiki constants of
  hyperk\"ahler manifolds}, arXiv:2201.07767, 2022.

\bibitem[Cam06]{cam06}
F.~Campana, \emph{Isotrivialit\'{e} de certaines familles k\"{a}hl\'{e}riennes
  de vari\'{e}t\'{e}s non projectives}, Math. Z. \textbf{252} (2006), no.~1,
  147--156.

\bibitem[dCRS21]{decat-rap-sac21}
M.~A. de~Cataldo, A.~Rapagnetta, and G.~Sacc\`a, \emph{The {H}odge numbers of
  {O}'{G}rady 10 via {N}g\^{o} strings}, J. Math. Pures Appl. (9) \textbf{156}
  (2021), 125--178.

\bibitem[Del71]{deligne71}
P.~Deligne, \emph{Th\'{e}orie de {H}odge. {II}}, Inst. Hautes \'{E}tudes Sci.
  Publ. Math. (1971), no.~40, 5--57.

\bibitem[Del72]{deligne72}
\bysame, \emph{La conjecture de {W}eil pour les surfaces {$K3$}}, Invent. Math.
  \textbf{15} (1972), 206--226.

\bibitem[Dol03]{dol:inv}
I.~Dolgachev, \emph{Lectures on invariant theory}, London Mathematical Society
  Lecture Note Series, vol. 296, Cambridge University Press, Cambridge, 2003.

\bibitem[FC90]{fal-chai:abelian}
G.~Faltings and C.-L. Chai, \emph{Degeneration of abelian varieties},
  Ergebnisse der Mathematik und ihrer Grenzgebiete (3), vol.~22,
  Springer-Verlag, Berlin, 1990.

\bibitem[Fis76]{fis:cg}
G.~Fischer, \emph{Complex analytic geometry}, Lecture Notes in Mathematics,
  Vol. 538, Springer-Verlag, Berlin-New York, 1976.

\bibitem[Fuj83]{fuj83}
A.~Fujiki, \emph{On primitively symplectic compact {K}\"{a}hler {$V$}-manifolds
  of dimension four}, Classification of algebraic and analytic manifolds
  ({K}atata, 1982), Progr. Math., vol.~39, Birkh\"{a}user Boston, Boston, MA,
  1983, pp.~71--250.

\bibitem[GGK19]{greb-gue-keb19}
D.~Greb, H.~Guenancia, and S.~Kebekus, \emph{Klt varieties with trivial
  canonical class: holonomy, differential forms, and fundamental groups}, Geom.
  Topol. \textbf{23} (2019), no.~4, 2051--2124.

\bibitem[GKKP11]{gkkp11}
D.~Greb, S.~Kebekus, S.~Kov\'{a}cs, and T.~Peternell, \emph{Differential forms
  on log canonical spaces}, Publ. Math. Inst. Hautes \'{E}tudes Sci. (2011),
  no.~114, 87--169.

\bibitem[GKLR22]{gklr22}
M.~Green, Y.-J. Kim, R.~Laza, and C.~Robles, \emph{The {LLV} decomposition of
  hyper-{K}\"{a}hler cohomology (the known cases and the general conjectural
  behavior)}, Math. Ann. \textbf{382} (2022), no.~3-4, 1517--1590.

\bibitem[GKP16]{greb-keb-pet16}
D.~Greb, S.~Kebekus, and T.~Peternell, \emph{Singular spaces with trivial
  canonical class}, Minimal models and extremal rays ({K}yoto, 2011), Adv.
  Stud. Pure Math., vol.~70, Math. Soc. Japan, [Tokyo], 2016, pp.~67--113.

\bibitem[GL14]{greb-lehn14}
D.~Greb and C.~Lehn, \emph{Base manifolds for {L}agrangian fibrations on
  hyperk\"{a}hler manifolds}, Int. Math. Res. Not. IMRN (2014), no.~19,
  5483--5487.

\bibitem[GTZ13]{gross-tos-zhang13}
M.~Gross, V.~Tosatti, and Y.~Zhang, \emph{Collapsing of abelian fibered
  {C}alabi-{Y}au manifolds}, Duke Math. J. \textbf{162} (2013), no.~3,
  517--551.

\bibitem[HT13]{has-tsc13}
B.~Hassett and Y.~Tschinkel, \emph{Hodge theory and {L}agrangian planes on
  generalized {K}ummer fourfolds}, Mosc. Math. J. \textbf{13} (2013), no.~1,
  33--56, 189.

\bibitem[Huy99]{huy99}
D.~Huybrechts, \emph{Compact hyper-{K}\"{a}hler manifolds: basic results},
  Invent. Math. \textbf{135} (1999), no.~1, 63--113.

\bibitem[Huy16]{huy:k3}
\bysame, \emph{Lectures on {K}3 surfaces}, Cambridge studies in advanced
  mathematics 158, Cambridge University Press, 2016.

\bibitem[Hwa08]{hwang08}
J.-M. Hwang, \emph{Base manifolds for fibrations of projective irreducible
  symplectic manifolds}, Invent. Math. \textbf{174} (2008), no.~3, 625--644.

\bibitem[Ive86]{ive:sheaves}
B.~Iversen, \emph{Cohomology of sheaves}, Universitext, Springer-Verlag,
  Berlin, 1986.

\bibitem[Kal06]{kal06}
D.~Kaledin, \emph{Symplectic singularities from the {P}oisson point of view},
  J. Reine Angew. Math. \textbf{600} (2006), 135--156.

\bibitem[Kaw85]{kaw85}
Y.~Kawamata, \emph{Minimal models and the {K}odaira dimension of algebraic
  fiber spaces}, J. Reine Angew. Math. \textbf{363} (1985), 1--46.

\bibitem[KM98]{kol-mori}
J.~Koll\'{a}r and S.~Mori, \emph{Birational geometry of algebraic varieties},
  Cambridge Tracts in Mathematics, vol. 134, Cambridge University Press,
  Cambridge, 1998.

\bibitem[KM18]{kap-men18}
S.~Kapfer and G.~Menet, \emph{Integral cohomology of the generalized {K}ummer
  fourfold}, Algebr. Geom. \textbf{5} (2018), no.~5, 523--567.

\bibitem[Kol96]{kol:ratcurve}
J.~Koll\'{a}r, \emph{Rational curves on algebraic varieties}, Ergebnisse der
  Mathematik und ihrer Grenzgebiete. 3. Folge. A Series of Modern Surveys in
  Mathematics, vol.~32, Springer-Verlag, Berlin, 1996.

\bibitem[KS90]{kas-sch:sheaves}
M.~Kashiwara and P.~Schapira, \emph{Sheaves on manifolds}, Grundlehren der
  Mathematischen Wissenschaften, vol. 292, Springer-Verlag, Berlin, 1990.

\bibitem[KV98]{kal-ver98}
D.~Kaledin and M.~Verbitsky, \emph{Partial resolutions of {H}ilbert type,
  {D}ynkin diagrams, and generalized {K}ummer varieties}, arXiv:9812078, 1998.

\bibitem[Laz04]{laz1}
R.~Lazarsfeld, \emph{Positivity in algebraic geometry. {I}}, Ergebnisse der
  Mathematik und ihrer Grenzgebiete. 3. Folge. A Series of Modern Surveys in
  Mathematics, vol.~48, Springer-Verlag, Berlin, 2004.

\bibitem[LL97]{loo-lunts97}
E.~Looijenga and V.~A. Lunts, \emph{A {L}ie algebra attached to a projective
  variety}, Invent. Math. \textbf{129} (1997), no.~2, 361--412.

\bibitem[LP93]{lepot93}
J.~Le~Potier, \emph{Faisceaux semi-stables de dimension {$1$} sur le plan
  projectif}, Rev. Roumaine Math. Pures Appl. \textbf{38} (1993), no.~7-8,
  635--678.

\bibitem[Mar14]{mar14}
E.~Markman, \emph{Lagrangian fibrations of holomorphic-symplectic varieties of
  {$K3^{[n]}$}-type}, Algebraic and complex geometry, Springer Proc. Math.
  Stat., vol.~71, Springer, Cham, 2014, pp.~241--283.

\bibitem[Mat15]{mat15}
D.~Matsushita, \emph{On base manifolds of {L}agrangian fibrations}, Sci. China
  Math. \textbf{58} (2015), no.~3, 531--542.

\bibitem[Mat16]{mat16}
\bysame, \emph{On deformations of {L}agrangian fibrations}, K3 surfaces and
  their moduli, Progr. Math., vol. 315, Birkh\"{a}user/Springer, 2016,
  pp.~237--243.

\bibitem[MB]{mathoverflow:mor}
L.~Moret-Bailly, \emph{To what extent does a torsor determine a group},
  MathOverflow, \url{https://mathoverflow.net/q/290094}.

\bibitem[Men14]{menet14}
G.~Menet, \emph{Duality for relative {P}rymians associated to {K}3 double
  covers of del {P}ezzo surfaces of degree 2}, Math. Z. \textbf{277} (2014),
  no.~3-4, 893--907.

\bibitem[Men20]{menet20}
\bysame, \emph{Global {T}orelli theorem for irreducible symplectic orbifolds},
  J. Math. Pures Appl. (9) \textbf{137} (2020), 213--237.

\bibitem[MFK94]{mum:git}
D.~Mumford, J.~Fogarty, and F.~Kirwan, \emph{Geometric invariant theory}, third
  ed., Ergebnisse der Mathematik und ihrer Grenzgebiete (2), vol.~34,
  Springer-Verlag, Berlin, 1994.

\bibitem[Mil80]{milne:etale}
J.~Milne, \emph{\'{E}tale cohomology}, Princeton Mathematical Series, No. 33,
  Princeton University Press, Princeton, N.J., 1980.

\bibitem[MO22]{mon-ono22}
G.~Mongardi and C.~Onorati, \emph{Birational geometry of irreducible
  holomorphic symplectic tenfolds of {O}'{G}rady type}, Math. Z. \textbf{300}
  (2022), no.~4, 3497--3526.

\bibitem[MR21]{mon-rap21}
G.~Mongardi and A.~Rapagnetta, \emph{Monodromy and birational geometry of
  {O}'{G}rady's sixfolds}, J. Math. Pures Appl. (9) \textbf{146} (2021),
  31--68.

\bibitem[MT07]{mar-tik07}
D.~Markushevich and A.~S. Tikhomirov, \emph{New symplectic {$V$}-manifolds of
  dimension four via the relative compactified {P}rymian}, Internat. J. Math.
  \textbf{18} (2007), no.~10, 1187--1224.

\bibitem[MW17]{mon-wan17}
G.~Mongardi and M.~Wandel, \emph{Automorphisms of {O}'{G}rady's manifolds
  acting trivially on cohomology}, Algebr. Geom. \textbf{4} (2017), no.~1,
  104--119.

\bibitem[Nag05]{nagai05}
Y.~Nagai, \emph{Dual fibration of a projective lagrangian fibration},
  Unpublished thesis, 2005.

\bibitem[Nam01]{nam01}
Y.~Namikawa, \emph{A note on symplectic singularities}, arXiv:math/0101028,
  2001.

\bibitem[Nam06]{nam06}
\bysame, \emph{On deformations of {$\mathbb Q$}-factorial symplectic
  varieties}, J. Reine Angew. Math. \textbf{599} (2006), 97--110.

\bibitem[Nit05]{nit:fga}
N.~Nitsure, \emph{Construction of {H}ilbert and {Q}uot schemes}, Fundamental
  algebraic geometry, Math. Surveys Monogr., vol. 123, Amer. Math. Soc.,
  Providence, RI, 2005, pp.~105--137.

\bibitem[Ogu09]{ogu09picard}
K.~Oguiso, \emph{Picard number of the generic fiber of an abelian fibered
  hyperk\"{a}hler manifold}, Math. Ann. \textbf{344} (2009), no.~4, 929--937.

\bibitem[Ogu20]{ogu20}
\bysame, \emph{No cohomologically trivial nontrivial automorphism of
  generalized {K}ummer manifolds}, Nagoya Math. J. \textbf{239} (2020),
  110--122.

\bibitem[Pou69]{pou69}
G.~Pourcin, \emph{Th\'{e}or\`eme de {D}ouady au-dessus de {$S$}}, Ann. Scuola
  Norm. Sup. Pisa Cl. Sci. (3) \textbf{23} (1969), 451--459.

\bibitem[Pri67]{prill67}
D.~Prill, \emph{Local classification of quotients of complex manifolds by
  discontinuous groups}, Duke Math. J. \textbf{34} (1967), 375--386.

\bibitem[PS08]{pet-ste}
C.~Peters and J.~Steenbrink, \emph{Mixed {H}odge structures}, Ergebnisse der
  Mathematik und ihrer Grenzgebiete. 3. Folge. A Series of Modern Surveys in
  Mathematics, vol.~52, Springer-Verlag, Berlin, 2008.

\bibitem[Rap07]{rap07}
A.~Rapagnetta, \emph{Topological invariants of {O}'{G}rady's six dimensional
  irreducible symplectic variety}, Math. Z. \textbf{256} (2007), no.~1, 1--34.

\bibitem[Rap08]{rap08}
\bysame, \emph{On the {B}eauville form of the known irreducible symplectic
  varieties}, Math. Ann. \textbf{340} (2008), no.~1, 77--95.

\bibitem[Sac20]{sacca20}
G.~Sacc\`a, \emph{Birational geometry of the intermediate {J}acobian fibration
  of a cubic fourfold}, arXiv:2002.01420, 2020.

\bibitem[Saw03]{sawon03}
J.~Sawon, \emph{Abelian fibred holomorphic symplectic manifolds}, Turkish
  Journal of Mathematics \textbf{27} (2003), no.~1, 197--230.

\bibitem[Saw04]{sawon04}
\bysame, \emph{Derived equivalence of holomorphic symplectic manifolds},
  Algebraic structures and moduli spaces, CRM Proc. Lecture Notes, vol.~38,
  Amer. Math. Soc., Providence, RI, 2004, pp.~193--211.

\bibitem[Saw09]{sawon09}
\bysame, \emph{Deformations of holomorphic {L}agrangian fibrations}, Proc.
  Amer. Math. Soc. \textbf{137} (2009), no.~1, 279--285.

\bibitem[Saw20]{sawon20}
\bysame, \emph{Lagrangian fibrations by {P}rym varieties}, Mat. Contemp.
  \textbf{47} (2020), 182--227.

\bibitem[Sch20]{sch20}
M.~Schwald, \emph{Fujiki relations and fibrations of irreducible symplectic
  varieties}, \'{E}pijournal G\'{e}om. Alg\'{e}brique \textbf{4} (2020), Art.
  7, 19 pp.--18.

\bibitem[Ser79]{serre:local_fields}
J.-P. Serre, \emph{Local fields}, Graduate Texts in Mathematics, vol.~67,
  Springer-Verlag, New York-Berlin, 1979.

\bibitem[{Sta}]{stacks-project}
The {Stacks Project Authors}, \emph{\textit{Stacks Project}},
  \url{https://stacks.math.columbia.edu}.

\bibitem[SY22]{shen-yin22}
J.~Shen and Q.~Yin, \emph{Topology of {L}agrangian fibrations and {H}odge
  theory of hyper-{K}\"{a}hler manifolds}, Duke Math. J. \textbf{171} (2022),
  no.~1, 209--241, With Appendix B by Claire Voisin.

\bibitem[SYZ96]{str-yau-zas96}
A.~Strominger, S.-T. Yau, and E.~Zaslow, \emph{Mirror symmetry is
  {$T$}-duality}, Nuclear Phys. B \textbf{479} (1996), no.~1-2, 243--259.

\bibitem[Ver95]{ver95}
M.~Verbitsky, \emph{Cohomology of compact hyperk\"ahler manifolds}, Ph.D.
  thesis, Harvard University, 1995.

\bibitem[Ver99]{ver99:mirror}
\bysame, \emph{Mirror symmetry for hyper-{K}\"{a}hler manifolds}, Mirror
  symmetry, {III} ({M}ontreal, {PQ}, 1995), AMS/IP Stud. Adv. Math., vol.~10,
  Amer. Math. Soc., Providence, RI, 1999, pp.~115--156.

\bibitem[vGV16]{vgee-voi16}
B.~van Geemen and C.~Voisin, \emph{On a conjecture of {M}atsushita}, Int. Math.
  Res. Not. IMRN (2016), no.~10, 3111--3123.

\bibitem[Voi92]{voi92}
C.~Voisin, \emph{Sur la stabilité des sous-variétés lagrangiennes des
  variétés symplectiques holomorphes}, London Mathematical Society Lecture
  Note Series, p.~294–303, Cambridge University Press, 1992.

\bibitem[Voi18]{voi18}
\bysame, \emph{Torsion points of sections of {L}agrangian torus fibrations and
  the {C}how ring of hyper-{K}\"{a}hler manifolds}, Geometry of moduli, Abel
  Symp., vol.~14, Springer, Cham, 2018, pp.~295--326.

\bibitem[Wie16]{wie16}
B.~Wieneck, \emph{On polarization types of {L}agrangian fibrations},
  Manuscripta Math. \textbf{151} (2016), no.~3-4, 305--327.

\bibitem[Wie18]{wie18}
\bysame, \emph{Monodromy invariants and polarization types of generalized
  {K}ummer fibrations}, Math. Z. \textbf{290} (2018), no.~1-2, 347--378.

\bibitem[Yos01]{yos01}
K.~Yoshioka, \emph{Moduli spaces of stable sheaves on abelian surfaces}, Math.
  Ann. \textbf{321} (2001), no.~4, 817--884.

\end{thebibliography}
\end{document}